\documentclass[reqno]{amsart}
\usepackage{amsmath}
\usepackage{amssymb}
  \usepackage{paralist}
  \usepackage{graphics} %% add this and next lines if pictures should be in esp format
 \usepackage[colorlinks=true]{hyperref}
\usepackage{mathrsfs}
\usepackage{extarrows}
\usepackage{enumerate}

\usepackage{tikz}
\usetikzlibrary{calc,trees,positioning,arrows,chains,shapes.geometric,patterns,%
	decorations.pathreplacing,decorations.pathmorphing,shapes,%
	matrix,shapes.symbols}

\tikzset{
	>=stealth',
	punktchain/.style={
		rectangle, 
		rounded corners, 
		draw=black, very thick,
		text width=9em, 
		minimum height=5em, 
		text centered, 
		on chain},
	line/.style={draw, thick, <-},
	element/.style={
		tape,
		top color=white,
		bottom color=blue!50!black!60!,
		minimum width=5em,
		draw=blue!40!black!90, very thick,
		text width=9em, 
		minimum height=5em, 
		text centered, 
		on chain},
	every join/.style={->, thick,shorten >=1pt},
	decoration={brace},
	tuborg/.style={decorate},
	tubnode/.style={midway, right=2pt},
}

\newcommand{\dmph}[1]{}

\newtheorem{theorem}{Theorem}[section]
\newtheorem{corollary}[theorem]{Corollary}

\newtheorem{lemma}[theorem]{Lemma}
\newtheorem{proposition}[theorem]{Proposition}
\newtheorem*{conjecture}{Conjecture}

\theoremstyle{definition}
\newtheorem{definition}[theorem]{Definition}
\newtheorem*{remark}{Remark}

\newtheorem{example}[theorem]{Example}
\newcommand{\ep}{\varepsilon}

\newcommand{\ext}[1]{\mathrm{ext}\left(#1\right)}
\newcommand{\RR}{\mathbb{R}}

\newcommand{\NN}{\mathbb{N}}

\newcommand{\ZZ}{\mathbb{Z}}

\newcommand{\cd}{\mathcal{D}}

\newcommand{\cS}{\mathcal{S}}

\newcommand{\cg}{\mathcal{G}}
\newcommand{\cP}{\mathcal{P}}

\newcommand{\cE}{\mathcal{E}}

\newcommand{\cB}{\mathcal{B}}
\newcommand{\cU}{\mathcal{U}}

\newcommand{\cH}{\mathcal{H}}
\newcommand{\cK}{\mathcal{K}}

\newcommand{\cm}{\mathcal{M}}

\newcommand{\sA}{\mathscr{A}}

\newcommand{\sss}{\mathscr{S}}
\newcommand{\sC}{\mathscr{C}}

\newcommand{\cM}{\mathcal{M}}

\newcommand{\spp}{\mathscr{P}}

\newcommand{\sv}{\mathscr{V}}

\newcommand{\sg}{\mathscr{G}}

\newcommand{\al}{\alpha}

\newcommand{\var}{\mathrm{Var}}

\title[Intermediate entropies and approximate
      product property]
      {Ergodic measures of intermediate entropies for dynamical systems with the approximate product property}

\author[Peng Sun]{}


 \email{sunpeng@cufe.edu.cn}

\begin{document}

\maketitle\ 

%\bigskip

% Enter the first author's name and address:
\centerline{\scshape Peng Sun}
\medskip
{\footnotesize
% please put the address of the first author
 \centerline{China Economics and Management Academy}
   \centerline{Central University of Finance and Economics}
   \centerline{Beijing 100081, China}
} % Do not forget to end the {\footnotesize by the sign }

\bigskip

%+Abstract
\begin{abstract}
	For a dynamical system satisfying the approximate product property and asymptotically
	entropy expansiveness, we characterize a delicate structrue of the space of invariant measures: The ergodic measures of intermediate entropies and intermediate pressures
	are generic in certain subspaces. This proves a conjecture of Katok for a broad class
	of systems and extends a sequence of known results.

%%%%%%%
\iffalse	
	It follows that a conjecture of Katok 
	holds for these systems.
	Our results also extend a sequence of known results and provide a uniform
	explanation for them.

We prove Katok's conjecture on existence of
ergodic measures with  arbitrary intermediate entropies for 
%Consider an
asymptotically entropy expansive systems with the approximate product property.
We characterize a fine structure that 
%ergodic 
these measures
%of intermediate entropies
are generic in certain compact subspaces.
%This 
%verifies a conjecture of Katok for these systems and 
%also extends
%the result of Sigmund that ergodic measures of zero entropy
%are generic.

%We consider dynamical systems with the approximate product property.
We show that 
every asymptotically entropy expansive systems with the approximate product property
%if in addition the system is asymptotically entropy expansive, then it 
has  ergodic measures of arbitrary intermediate entropies and arbitrary intermediate pressures,
which are actually generic in certain compact subspaces. 
The result verifies  a conjecture of Katok for a 
 broad class of systems
%In particular, this implies that the ergodic measures of 
%zero  entropy
%are generic, which 
and extends a result of Sigmund 
on the genericity of the ergodic measures of 
zero  entropy.
%are generic,
%in a 
%much generalized
%setting where periodic points may be absent.
\fi
%%%%%%%%%%%%%
\end{abstract}
%-Abstract

%+Contents
%\tableofcontents
%-Contents

\section{Introduction}

It is a question with a long history  whether positive topological entropy
implies a rich structure of the space of invariant measures.
%that the space of invariant measures has a rich structure.
Parry asked if a strictly ergodic (i.e. minimal and uniquely ergodic) system must have zero topological entropy.
The answer is negative and %there are 
many $C^0$ counterexamples have been found
% with positive topological entropy that are uniquely ergodic
%(e.g. \cite{HK}, \cite{GW} and \cite{BC}). 
(e.g. \cite{BC,GW,HK}). 
However, it seems that we may expect a positive answer for smooth systems, 
as conjectured by Herman,
because in this case positive topological entropy implies 
existence of nonzero Lyapunov exponents, from which we
can obtain
some hyperbolic structure.
%sort of 
%(weak) hyperbolic structure
%hyperbolicity. 
In the
seminal work \cite{Ka80} Katok showed that for
$C^{1+\al}$ diffeomorphisms
in dimension 2, positive topological entropy implies existence of horseshoes.
It %also 
follows that the system
has ergodic measures of arbitrary intermediate metric entropies. Katok believed
that this holds for any smooth system (of sufficient regularity) in any dimension.

\begin{conjecture}[Katok]
For every $C^2$ diffeomorphism $f$ on a compact Riemannian manifold $X$, the set
$$\cH(X,f):=\bigl\{h_\mu(f): \mu\text{ is an ergodic measure for }(X,f)\bigr\}$$
includes $\bigl[0,h(f)\bigr)$.
\end{conjecture}

We say that a system has the \emph{intermediate entropy property}
 if it verifies
Katok's conjecture.
Progress on the conjecture was made by the author \cite{Sun09,
Sun10, Sun12} for certain skew products and toral automorphisms. 
Ures \cite{Ures} and Yang and Zhang
\cite{YZ}  proved
the conjecture for some
partially hyperbolic diffeomorphisms
with one-dimensional center. 
In the remarkable work \cite{QS}, Quas and Soo showed that
a system is universal, which implies %the conclusion of 
%Katok's conjeture,
the intermediate entropy property,
if it satisfies asymptotic entropy expansiveness, the almost weak specification
property
and the small boundary property. 
%Recently,
%extended 
%The result of Quas and Soo was then extended 
Their result was improved
by Burguet \cite{Burg} 
who showed that the almost weak specification property itself is sufficient
for universality.
In the joint work \cite{GSW} of %the author with %Lifan 
Guan, Wu and the author,
% and %Weisheng Wu, 
we showed
that certain homogeneous systems has the almost weak specification, hence
proved Katok's conjecture by applying the previous results.
However, universality is not implied by the approximate product property.
See Example \ref{nouniver}.
%However, in general, universality should not be expected for systems with only the approximate
%product property (e.g. Example \ref{nouniver}).  
In \cite{KKK}, Konieczny, Kupsa and Kwietniak proved Katok's conjecture
for hereditary shifts,
%who showed 
by showing that the set of ergodic measures is arcwise connected with respect to
a special metric under which the entropy function is continuous.
Recently, the author has developed an approach based on uniqueness
of equilibrium states and proved Katok's conjecture for a class of
Ma\~n\'e systems \cite{Sun21}.
%We have also learned about the 
Moreover, a flow version of Katok's conjecture
has been proved by Li, Shi, Wang and Wang for star flows \cite{LSWW}.
%a work %notice that
%the results \cite{YZ} by Yang and Zhang for certain partia and 
%\cite{LSWW} by
%Li, Shi, Wang and Wang on a flow version of Katok's conjecture 
%for star flows.
%After the first manuscript was , we note that 

In this article, we prove Katok's conjecture for
asymptotic entropy expansive  systems 
%satisfying  
with the approximate product property.
We  actually characterize a delicate %characterization of the 
structure of the space of invariant measures,
which presents much stronger conclusions than intermediate entropy property. 
%and asymptotic entropy expansiveness.
The 
approximate product property was first introduced by Pfister and Sullivan
\cite{PfSu}, which is almost the weakest one 
among the so-called \emph{specification-like
properties} \cite{KLO}. 
%It was Quas and Soo's work that 
%first
%brought our attention to 
%the connection between such properties and Katok's conjecture.
%the so-called \emph{specification-like properties}.
%Such properties are topological properties
%originally abstracted from certain hyperbolicity.
%The specification-like properties are closely related to hyperbolicity.
Various orbit-tracing properties are closely related to hyperbolicity and have
played import roles in studying smooth systems.
%Research on systems with certain hyperbolicity has been a mainstream direction in the
%field of dynamical systems, in which various orbit-tracing properties play import roles. 
In 1971, Bowen introduced the notion of specification
%in his seminal work \cite{Bowen} 
to study periodic points and
invariant measures for Axiom A diffeomorphisms \cite{Bowen}.
Since then, a number of variations of the specification property have been 
introduced
to study broader classes of dynamical systems, which represent weaker forms of
hyperbolicity. Based on these specification-like properties,  a bunch of 
interesting results were
successfully achieved.
%, including a sequence of works
%\cite{Sun19, Sunct, Sunze, Sun21, Sununierg, Sunnd}
%by the author.
The author have also witnessed the power of these properties in a sequence
of works %a sequence of works
\cite{Sun19, Sunct, Sunze, Sununierg, Sunnd}.
%hence have enough reasons to
%expect 

%, including the ones on the
%structure of the space of invariant measures, 
%growth rate of periodic orbits, large deviations, multifractal analysis,
%etc. 
%For more definitions and results of specification-like
%properties, the readers are referred to %the book 
%\cite{DGS, KLO, Sunze}.
% and the survey \cite{KLO}. %by Kwietniak, Lacka and Oprocha.

%By \emph{topological dynamical system}
%or just \emph{system}, we mean an order
%on the structure of the space of invariant measures.
%In particular,
%this implies that the set of ergodic measures of zero entropy is residual, which
%extends a theorem of Sigmund \cite{Sig}. 
%We remark that, compared with previous
%results, in our setting 
%the system may have no periodic points.
%Our result indicates that we might be able to ask
%for something more than Katok's conjecture for a broad class of systems.
Let $X$ be a compact metric space and $f:X\to X$ be a continuous map. 
%For the system 
Then we say that $(X,f)$ is %usually called 
a \emph{topological dynamical system}
or just a \emph{system}.
%, d
Denote by $\cm(X,f)$ the space of its invariant measures
and by $\cm_e(X,f)$ the subset of ergodic ones. 
%Let $\Lambda$ be a compact
%subset of $X$ such that $f(\Lambda)\subset\Lambda$. Let
%$\phi:X\to\RR$ be a continuous potential function.
Let %$\Phi:X\to\RR$ 
$\Phi$ be an asymptotically additive 
%a continuous 
potential for $(X,f)$ (see Definition \ref{defaapot}).
Denote by $h(f)=h(X,f)$ and $P(f,\Phi)=P(X,f,\Phi)$ the topological entropy
and the topological pressure.
% and . 
Denote by $h_\mu(f)$  and
$P_\Phi(\mu)$
%:=h_\mu(f)+\int\phi d\mu$ 
the entropy
and the pressure
of $\mu\in\cm(X,f)$.
% of $\mu$. 
Let
$$P_{\inf}(f,\Phi):=\inf\bigl\{P_\Phi(\mu):\mu\in\cm(X,f)\bigr\}.$$
A subset $\Lambda$ of $X$ is called \emph{$f$-invariant} if $f(\Lambda)\subset\Lambda$.
When $\Lambda$ is compact and $f$-invariant, $(\Lambda,f)$ is also a
topological dynamical system, hence the notations like $\cm(\Lambda,f)$
%and $h(\Lambda,f)$ 
make sense.
%We shall show that the following
%properties hold. 
For convenience, please allow us to assign names to %these 
the following properties.
The names are inspired by the terminologies 
\emph{entropy-approachable}
and \emph{entropy-dense} introduced in
\cite[Definition 2.7]{PfSu}.
%\cite{PfSu}.

\begin{definition}\label{defea}
Let $(X,f)$ be a topological dynamical system.
% and 
%$\Phi$ be an asymptotically additive potential for $(X,f)$.
%$\phi:X\to\RR$ be %an asymptotically additive 
%a continuous potential. %for $(X,f)$..
\begin{enumerate}
\item Given $\mu\in\cm(X,f)$, we say that $\mu$ is \emph{almost
entropy-approximable} (by compact invariant sets of intermediate entropies), if
for every neighborhood $U$ of $\mu$, every $h\in\bigl(0,h_\mu(f)\bigr)$ and every $\ep, \beta>0$, 
there are a compact $f$-invariant
set $\Lambda$ and $\gamma\in(0,\ep)$ such that
$$\cm(\Lambda, f)\subset U, \; h(\Lambda,f)>h\text{ and }
h(\Lambda, f, \gamma)<h+\beta,$$
%where $h(\Lambda, f, \gamma)$ denotes the topological
%entropy of the subsystem $(\Lambda, f)$ calculated at the scale $\gamma$.
%and $|
\item Given $\mu\in\cm(X,f)$, we say that $\mu$ is \emph{entropy-approximable} 
(by compact invariant sets of intermediate entropies), if
for every neighborhood $U$ of $\mu$, every $h\in\bigl(0,h_\mu(f)\bigr)$ and every $\beta>0$,
there is a compact $f$-invariant
set $\Lambda$ such that
$$\cm(\Lambda, f)\subset U\text{ and }
%$$
h<h(\Lambda, f)<h+\beta.$$
\item We say that the system $(X,f)$ is \emph{entropy-generic}, if
for every $\al\in\bigl[0, h(f)\bigr)$, the set
        $$\cm_e(X,f,\al):=\bigl\{\mu\in\cm_e(X,f): h_\mu(f)=\alpha\bigr\}$$
        is residual %subset 
        in the %compact metric 
        subspace
        $$%\cm^\al=
        \cm^\al(X,f):=\bigl\{\mu\in\cm(X,f): h_\mu(f)\ge\alpha\bigr\}.$$
\item 
%Let $\Phi$ be an asymptotically additive  potential for $(X,f)$.
We say that %the system 
        $(X,f,\Phi)$ is \emph{pressure-generic}, if
%        for every asymptotically additive 
%potential $\Phi$ for $(X,f)$ and 
        for every $\al\in\bigl(P_{\inf}(f,\Phi), P(f,\Phi)\bigr)$, the set
        $$\spp_e(X,f,\Phi,\al):=
        \bigl\{\mu\in\cm_e(X,f): P_\Phi(\mu)=\alpha\bigr\}$$
        is residual %subset 
        in the %compact metric 
        subspace
        $$\spp^\al(X,f,\Phi):=\{\mu\in\cm(X,f):\chi_\Phi(\mu)\le\al\le P_\Phi(\mu)\}.$$
%        $$%\cm^\al=
 %       \spp^\al(X,f,\Phi):=
  %      \left\{\mu\in\cm(X,f): 
   %     \int\phi d\mu\le\alpha\le P_\Phi(\mu)\right\}.$$        
%\item For $U\subset\cm(X,f)$, denote
%$$\cH(X,f,U):=\{h_\nu(f): \nu\in U\cap\cm_e(X,f)\}.$$
%        Then we have
%We say that the system $(X,f)$ 
%\emph{has ergodic measures of intermediate entropies in every neighborhood} if
%for every $\mu\in\cm(X,f)$ and every neighborhood $U$ of $\mu$, we have
%        $$\begin{cases}
%        \cH(X,f,U)\supset[0, h_\mu(f)], & \text{if }h_\mu(f)<h(f);\\
%        \cH(X,f,U)\supset[0, h_\mu(f)),¨Í¨Í & \text{if }h_\mu(f)=h(f).
%        \end{cases}$$
        
\end{enumerate}
\end{definition}

\dmph{We remark that pressure-genericity may fail when $\al=P_{\inf}(f,\Phi)$.
 So it does not imply entropy-genericity just by definition. See
Example \ref{expinf}.
}

%We assign names to the above properties for convenience.

%Now 
%We can state our key theorem as follows.

The following is our key theorem that only assumes the approximate product property.

\begin{theorem}%[{\cite[Proposition 5.1]{Sunintent}}]
\label{thaea}
Let $(X,f)$ be a system with the approximate product property.
%and finite topological .
Then every  invariant measure  $\mu\in\cm(X,f)$
is almost entropy-approximable.
\end{theorem}

%When the system is asymptotically entropy expansive, we have
%$$\lim_{\gamma\to 0}h(\Lambda, f, \gamma)=h(\Lambda, f)$$
%and both $\cm^\al(X,f)$ and $\spp^\al(X,f,\phi)$ are compact subspaces.
%Then 
Our main result %is a corollary of 
follows %from Theorem \ref{thaea} 
if in addition the system is
asymptotically entropy expansive.

\begin{theorem}\label{thmain}
Let $(X,f)$ be an asymptotically entropy expansive system
with the approximate product property.
%Let $\Phi:X\to\RR$ be an asymptotically additive 
%a continuous 
%potential for $(X,f)$.
%Suppose that $(X,f)$ is also asymptotically entropy expansive.
%Then 
Then the following hold:
\begin{enumerate}
\item\label{thit1} Every  invariant measure  $\mu\in\cm(X,f)$
is entropy-approximable.
\item\label{thit2} The system $(X,f)$ is entropy-generic.
\item\label{thit3} %The system 
Let $\Phi$ be any asymptotically additive potential for $(X,f)$.
Then $(X,f, \Phi)$ is pressure-generic.
%\item\label{thit4} The system $(X,f)$ 
%has ergodic measures of intermediate entropies in every neighborhood.
%\item\label{thit5} If there is $\mu_0\in\cm(X,f)$ such that
%$P_{\mu_0}(f,\phi)=P_{\inf}(f,\phi)$
%(such $\mu_0$ may not exist), 
%then $h_{\mu_0}(f)=0$, i.e. the minimum of the pressure function
%either does not exist or it can only be assumed at 
%invariant measures of zero entropy.
\end{enumerate}
\end{theorem}

%\begin{remark}
We remark that after the first preprint of the article was posted on arXiv, the author
was told that
Li and Oprocha \cite{LO} had obtained a similar result to 
Theorem \ref{thmain} \eqref{thit2} under a stronger assumption that the system
is topologically transitive
%systems with 
and has the shadowing property, and the entropy map is %should be
upper semi-continuous. They used a different method
focusing on odometers.
Another remark is that Cueno \cite{Cueno} has shown that every
asymptotically additive potential is equivalent
to a standard one. However, asymptotically additive condition
remains relevant for applications.

By investigating the structure of $\cm(X,f)$ characterized by
Theorem \ref{thmain}, we have also obtained the following corollaries:

\begin{corollary}\label{cor_nbhd}
Let $(X,f)$ be an asymptotically entropy expansive system
with the approximate product property. 
Let $\Phi$ be an asymptotically additive potential for $(X,f)$.
%Let $\mu\in\cm(X,f)$ and $U$ be a neighborhood of $\mu$.
For $U\subset\cm(X,f)$,
denote
$$\cH(X,f,U):=\bigl\{h_\nu(f): \nu\in U\cap\cm_e(X,f)\bigr\}$$
and
$$\cP(X,f, \Phi, U):=\{P_\Phi(\nu): \nu\in U\cap\cm_e(X,f)\}.$$
Then the following hold:
\begin{enumerate}
\item\label{itnb1} For every $\mu\in\cm(X,f)$ and every neighborhood $U$ of $\mu$, we have
%Then we have
$$\begin{cases}
        \cH(X,f,U)\supset\bigl[0, h_\mu(f)\bigr], & \text{if }h_\mu(f)<h(f);\\
        \cH(X,f,U)\supset\bigl[0, h_\mu(f)\bigr), & \text{if }h_\mu(f)=h(f).
        \end{cases}$$
\item\label{itnb2} For every $\mu\in\cm(X,f)$ and every neighborhood $U$ of $\mu$, we have
$$\begin{cases}
        \cP(X,f, \Phi, U)\supset\bigl[\chi_\Phi(\mu), P_\Phi(\mu)\bigr],
         & \text{if }P_\Phi(\mu)<P(f,\Phi);\\
        \cP(X,f, \Phi, U)\supset\bigl[\chi_\Phi(\mu), P_\Phi(\mu)\bigr), & \text{if }P_\Phi(\mu)=P(f,\Phi),
        \end{cases}$$
        where $\chi_\Phi(\mu)$ is the Lyapunov exponent of $\mu$ (see Subsection 
        \ref{subile}).
\item\label{itnb3} If there is $\mu_0\in\cm(X,f)$ such that
        $P_\Phi(\mu_0)=P_{\inf}(f,\Phi)$,
        then we must have
        $h_{\mu_0}(f)=0$. That is, the infimum of $P_\Phi$
%either can not be obtained or
can either not be obtained, or only be obtained at 
a measure of zero entropy. 
\end{enumerate}

%For $U\subset\cm(X,f)$, denote
%$$\cH(X,f,U):=\bigl\{h_\nu(f): \nu\in U\cap\cm_e(X,f)\bigr\}.$$
%$$\cP(X,f, \Phi, U):=\{P_\Phi(\nu): \nu\in U\cap\cm_e(X,f)\}.$$
\end{corollary} 

We remark that it is possible that $h_\mu(f)\notin \cH(X,f,U)$ when $h_\mu(f)=h(f)$. This
happens if $(X,f)$ has multiple ergodic measures of maximal entropy and this
is compatible with the approximate product property. 
See Example \ref{nomaxent}.
It is clear that if in addition $(X,f)$ is intrinsically ergodic
(i.e. it has exactly one ergodic measure of maximal entropy) then
$$\cH(X,f,U)\supset[0, h_\mu(f)]\text{ for every }\mu\in\cm(X,f).$$
%We also remark that there may not exist a measure $\mu\in\cm(X,f)$
%satisfying
%        $P_\Phi(\mu)=P_{\inf}(f,\Phi)$.
%See Example \ref{expinf}.

%Conclusions \eqref{thit3} and \eqref{thit5} remain valid if the continuous
%potential $\phi$ is replaced by a sub-additive potential. %However,
%By \cite{Cueno}, the two cases are equivalent.
%\end{remark}

%As the approximate product property is almost 
%the weakest specification-like property,
%Theorem \ref{thaea} and \ref{thmain} apply to many
%well-known systems that span a broad spectrum of interests.
The approximate product property, as well as asymptotic entropy expansiveness,
is widely satisfied by 
many classical systems that span a broad spectrum of interests.
The following is %an incomplete 
a summarization of just a few
known results.

\begin{proposition}\label{pr_app}
The following systems has the approximate product property: 
%and are also asymptotically entropy expansive, to which the theorems \ref{} apply:
\begin{enumerate}[(i)]
\item\label{piti} Some symbolic systems, which are expansive, including:
\begin{itemize}

%\item Transitive subshifts of finite types
%, which have the gluing orbit property
%\cite{BV};
\item
Transitive sofic shifts, including all %topologically 
transitive
subshifts of finite types
%, which have the approximate product property
\cite[Corollary 40]{KLO};
\item $\beta$-shifts
%, which have the approximate product property 
\cite{PfSu};

\end{itemize}

\item\label{pitii} Some automorphisms on compact groups, which are asymptotically entropy expansive,
%by Proposition
%\cite[Example 7.1]{Misiu}, 
including:
\begin{itemize}
\item Ergodic toral automorphisms
%, which have periodic tempered specification property 
\cite{Marc};
\item An automorphism of a finite-dimensional compact metric abelian group with finite topological entropy,
whose Koopman
representation has no finite orbits on the character group other than the
trivial character
%, which has tempered specification property 
\cite[Theorem 11]{QS};
\item A homogeneneous system $(G/\Gamma,g)$, where $G$ is connected semisimple Lie group without compact factors,
$\Gamma$ is an irreducible cocompact lattice of G and $g\in G$ is non-quasiunipotent
%, which has the tempered specification property
\cite{GSW};
\end{itemize}
\item The restriction, of every $C^0$-generic map $f$ on a compact
Riemannian manifold, to a single chain-recurrent class $\Omega$ for $f$
%, 
%$f|_\Omega$
%which has the gluing orbit property 
\cite[Corollary 2]{BTV};

\item Transitive graph maps
%, which have the approximate product property
\cite[Corollary 40]{KLO};
%, if in addition they are asymptotically entropy
%expansive (e.g. $C^\infty$ maps);

\item Certain partially hyperbolic diffeomorphisms,
e.g. topologically transitive time-1 maps of Anosov flows;

\item A topologically transitive system with the shadowing property \cite{TS};

\item A product of a systems with the approximate product property and a system
with the tempered specification property (see Proposition \ref{prodapp});
%, if it is also
%asymptotically expansive (e.g. Example \ref{nouniver} and \ref{tempernogo});
\item Factors and conjugates of above systems (see Proposition \ref{aeesys} and Proposition
\ref{appfactor}).
%\item Some other partially hyperbolic systems (work in propress).
\end{enumerate}
\end{proposition}

%\begin{proposition}
%Let $X$ be a compact
%Riemannian manifold. Then
%$\cH(f,
%\end{proposition}

%Note that the
%systems of the categories \eqref{piti} and \eqref{pitii}
%in Proposition \ref{pr_app}
%are also asymptotically entropy expansive.
%So 
Theorem \ref{thaea} applies to all systems
listed in Proposition \ref{pr_app}.
%in the above list.
Theorem \ref{thmain} directly applies to 
%these systems directly.
 the systems of the categories \eqref{piti} and \eqref{pitii}.
%The theorem 
It applies to any other system in the list
%listed in Proposition \ref{pr_app}
if in addition the system is asymptotically entropy expansive.

%Both \eqref{thit2} and \eqref{thit4} of 
Theorem \ref{thmain} \eqref{thit2} implies that
%Katok's conjecture holds
the intermediate entropy property holds
for asymptotically
entropy expansive systems with the approximate product property.
%and asymptotically entropy expansiveness.
%Proposition \ref{pr_apply} provides a selected list
%of such systems, according to known results.
%Then %we can see that 
%our result 
%The results apply to the systems listed
Corollary \ref{cor_nbhd} \eqref{itnb1} shows further that
ergodic measures of intermediate
entropies exist in every neighborhood of
an invariant measure.
With Proposition \ref{pr_app}, %this %result
our results %can
has %actually 
covered many known results on Katok's conjecture, including 
\cite{GSW},
\cite[Corollary C (2)]{LO},
\cite[Section 3.2]{QS},
\cite{Sun12} and
\cite[Theorem 1.3]{Sunze}, %However,
%We believe that
%This also provides 
as well as providing
a uniform explanation for them.
Similar ideas can also be applied to study more systems,
e.g. \cite{Sunct} for systems admitting a Climenhaga-Thompson
decomposition.
%Entropy-genericity (resp. pressure-genericity)
%also imply that the system has ergodic measures of intermediate
%entropies (resp. intermediate pressures) in every neighborhood of
%an invariant measure. See Proposition
%\ref{localintent} and \ref{localintpr}.

In \cite{Sig}, Sigmund proved various generic properties of invariant measures
for Axiom A diffeomorphisms. Similar results are obtained in \cite{GK} for
systems satisfying certain properties related to periodic points.
Theorem \ref{thmain} \eqref{thit2} extends a result of Sigmund
to our setting that 
 $\cm_e(X,f,0)$, the
set of ergodic measures of zero entropy, is
residual in $\cm(X,f)$. This also provides a partial answer to a question raised in \cite{BV}
by Bomfim and Varandas for systems with the gluing orbit property.
We remark that, 
compared with previous
results, in our setting 
the system may have no periodic points.
%Note that the systems we consider may have no periodic orbits.

Let us get back to Parry's question.
In \cite{Sunze}, we have shown that a system has the gluing orbit
property and zero topological entropy if and only if it is strictly ergodic and equicontinuous.
We note that there are
% non-trivial 
subshifts 
(hence expansive and not equicontinuous) that has the approximate
product property and zero topological entropy while they 
are not even topologically transitive. 
%(
See Example \ref{exnoneq} and \ref{exnontrans}.
%).
%By Theorem \ref{thaea}, we are able to see 
In Subsection \ref{appminimal}, we shall prove the
following corollary:
\begin{corollary}\label{cor_min}
Let $(X,f)$ be a minimal system with the approximate product
property. Then $(X,f)$ must be uniquely ergodic and $h(f)=0$.
\end{corollary}

%Example \ref{exnoneq} and \ref{exnontrans} show that
%there are systems with the approximate product
%property that are of zero entropy but not minimal.

Further investigation in \cite{Sununierg} shows that
% show that for a system with the approximate product
%property, minimality implies zero topological entropy and unique ergodicity.
% (
%See Subsection \ref{appminimal}.
%). 
%Recently we have shown
%that for such a system, zero topological entropy is equivalent to unique ergodicity \cite{Sununierg}.
%Along with Theorem \ref{alresidual}, 
%This indicates 
%It follows that 
there is a dichotomy on the structure of $\cm(X,f)$,
%on the structure of the space of invariant measures 
for a system with the approximate product property, which is 
completely determined
by the topological entropy:
$$\begin{cases}
h(f)=0 \iff {}&{} \text{$\cm(X,f)$ is a singleton}.\\
h(f)>0 \iff {}&{} \text{$\cm(X,f)$ is a Poulsen simplex}.
%\\&\text{$\cm_e(X,f,\al)$ is a residual subset of $\cm^\al(X,f)$ for each }
\end{cases}$$
This 
%also 
complements Katok's conjecture in the case
that the system has zero topological entropy. 
\dmph{Furthermore, the fact that $\cm(X,f)$ is either a singleton or a Poulsen simplex
%also 
implies that $\cm_e(X,f)$ is arcwise connected. It follows that
every system with the approximate product property has intermediate
Lyapunov exponents for asymptotically additive
functions, which %greatly 
extends the result of \cite{TWW}.
See Subsection \ref{subile}.

%In \cite{Sunze} we have shown that for systems with the gluing orbit
%property, zero topological entropy  implies minimality and equicontinuity.
%Dong \cite{Dong} also has an example with zero topological entropy and almost specification property.
%We wonder which is the case for systems with tempered the gluing orbit property.

%We have found more corollaries of Theorem \ref{thmain}, which are just
%consequences of entropy-genericity and pressure-genericity.
}

Notions and results in this article naturally extends to the continuous-time case. The proof
can be carried out with a little extra effort, namely a discretization argument as in the proof of \cite[Lemma 5.10]{CLT}.

The article is organized as follows: 
We provide some preliminaries %are provided 
in Section \ref{secpre}.
Then we prepare some lemmas about empirical measures in Section \ref{secem}.
We %shall 
prove Theorem \ref{thaea} in Section \ref{ldsection}
and discuss its corollaries concerning minimality in Subsection \ref{appminimal}.
We prove Conclusion \eqref{thit1} and \eqref{thit2}
of Theorem \ref{thmain} in Section \ref{secintent}.
We discuss 
Lyapunov exponents and pressures for
%results related to
asymptotically additive potentials in Section \ref{secaap}
and prove Conclusion \eqref{thit3} %and \eqref{thit2}
of Theorem \ref{thmain} 
%and discuss results related to
%asymptotically additive potentials
in Subsection \ref{subpregen}.
Finally, we present
some examples related to our results %are presented 
in Section \ref{secexamp}.
%a key proposition
%(Proposition \ref{prmain}) in
%Subsection \ref{mainconstruction}
%and \ref{mainest}.

 \section{Preliminaries}\label{secpre}
%Let $(X,d)$ be a compact metric space. 
%Throughout this article, we assume that $X$ is
%infinite to avoid trivial exceptions.  
%Let $f:X\to X$ be a continuous map. 
%Then $(X,f)$ is conventionally called a \emph{topological dynamical system} or just a
%\emph{system}.
In what follows, we always assume that $(X,f)$ is a topological
dynamical system.
We shall denote by 
$\ZZ^+$ the set of all positive integers and by $\NN$ the set of all nonnegative integers, i.e.
$\NN=\ZZ^+\cup\{0\}$. %Let $L:\ZZ^+\to\ZZ^+$ be a non-decreasing function.
For $n\in\ZZ^+$,  denote
%Denote 
$$\ZZ_n:=\{0,1,\cdots, n-1\}\text{ and }\Sigma_n:=\{0,1,\cdots,n-1\}^{\ZZ^+}.$$
%$$Q_M:=\{1,2,\cdots, M\}\text{ and }\Sigma_M:=\{1,2,\cdots, M\}^{\ZZ^+}.$$
%Denote
%$$Q_n:=\{1,\cdots,n\}.$$
%we denote  by
%$$\Sigma_M:=\{1,2,\cdots, M\}^{\ZZ^+}$$
%the space of sequences in $\{1,2,\cdots, M\}$.
Readers may find more details on entropies and invariant measures
in \cite{Walters}.

\subsection{Topological entropy and expansiveness}
\begin{definition}
Let $K$ be a subset of $X$.
For $n\in\ZZ^+$ and $\ep>0$, a subset $E\subset K$ 
is called an \emph{$(n,\ep)$-separated set} in $K$ if
for any distinct points $x,y$ in $E$, 
we have
$$d_n^f(x,y):=\max\Bigl\{d\bigl(f^k(x),f^k(y)\bigr): k\in\ZZ_n\Bigr\}>\ep.$$
%there is $k\in\{0,\cdots,n-1\}$ such that
%$$d(f^k(x),f^k(y))>\ep.$$
Denote by $s(K,n,\ep)$ the maximal cardinality of an $(n,\ep)$-separated subset
of $K$. 
Let
$$h(K,f,\ep):=\limsup_{n\to\infty}\frac{\ln s(K,n,\ep)}{n}.$$
Then the \emph{topological entropy} of $f$ on $K$ is defined as
$$h(K,f):=\lim_{\ep\to0}h(K,f,\ep).$$
In particular, $h(f):=h(X,f)$ is the topological entropy of the system $(X,f)$.
%and usually denoted as $h(f)$.
\end{definition}

%\begin{remark}
For each $n\in\ZZ^+$, $d_n^f$ is a metric on $X$.
Note that $h(K,f,\ep)$ grows as $\ep$ tends to $0$. So we actually have
\begin{equation}\label{eqhsup}
h(K,f)=\sup\bigl\{h(K,f,\ep):\ep>0\bigr\}.
\end{equation}
%\end{remark}

%\begin{proposition}
%If $K$ is a compact subset of $X$  such that $f(K)\subset K$. Then
%$$h(K,f)=\lim_{\ep\to0}\liminf_{n\to\infty}\frac{\ln s(K,n,\ep)}{n}.$$
%\end{proposition}

\begin{definition}
Let $\ep>0$.
A set of the form
%$$B_n(x,\ep)=\{y\in X: d(f^k(y),f^k(x))<\ep, k=0,1,\cdots,n-1\}$$
$$B_n(x,\ep)=\bigl\{y\in X: d_n^f(x,y)<\ep\bigr\}$$
is called an \emph{$(n,\ep)$-ball} of $(X,f)$. 
A subset $E$ of $X$ is called an
\emph{$(n,\ep)$-spanning set} if
$$X=\bigcup_{x\in E}B_n(x,\ep).$$
%For $0<\ep<\frac13\gamma$, let $Y=Y(\tau,\ep)$ and

\end{definition}

Denote by $r(n,\ep)$ %=E(M(\ep),\ep)$ 
the minimal cardinality
of an $(n,%\frac12
\ep)$-spanning subset of $X$. In particular, we denote $r(\ep):=r(1,\ep)$.
By \cite[Lemma 2.1]{Bowen2}, we have
\begin{equation}\label{finent}
r(n,\ep)\le r(\ep)^n\text{ for every }n\in\ZZ^+.
\end{equation}
%We shall need this inequality for entropy estimate.

\dmph{
\begin{proposition}\label{entspan}
Let
$$h_s(f,\ep):=\limsup_{n\to\infty}\frac{\ln r(n,\ep)}{n}.$$
Then we have
$$h(f)=\lim_{\ep\to 0}h_s(f,\ep)=
\sup\bigl\{h_s(f,\ep):\ep>0\bigr\}.$$
%$$h(f)=\lim_{\ep\to 0}\limsup_{n\to\infty}\frac{\ln r(n,\ep)}{n}=\sup\{\limsup_{n\to\infty}\frac{\ln r(n,\ep)}{n}:\ep>0\}.$$
\end{proposition}

Given $\ep>0$, by \cite[Lemma 2.1]{Bowen2}, we have
\begin{equation}\label{finent}
r(n,\ep)\le r(\ep)^n\text{ for every }n\in\ZZ^+.
\end{equation}
It follows that %for any $\ep>0$, as $X$ is compact, we have
\begin{equation*}%\label{finiteent}
h_s(f,\ep)\le\ln r(\ep)<\infty.
\end{equation*}
}

\begin{definition}\label{defexp}
For $\ep>0$ and $x\in X$, denote
$$\Gamma_\ep(x):=\Bigl\{y\in X:
d\bigl(f^n(x),f^n(y)\bigr)<\ep\text{ for every }n\in\NN\Bigr\}.$$
Let
$$h^*(f,\ep):=\sup\Bigl\{h\bigl(\Gamma_\ep(x),f\bigr):x\in X\Bigr\}.$$
\begin{enumerate}
\item We say that $(X,f)$ is \emph{expansive} if there is $\ep_0>0$ such
that $\Gamma_{\ep_0}(x)=\{x\}$ for every $x\in X$.
\item We say that $(X,f)$ is \emph{entropy expansive} if there is $\ep_0>0$
such that $h^*(f,\ep_0)=0$.
\item 
We say that $(X,f)$ is \emph{asymptotically entropy expansive} if
$$\lim_{\ep\to 0}h^*(f,\ep)=0.$$
\end{enumerate}
\end{definition}

\begin{proposition}[cf. {\cite[Theorem 2.4]{Bowen2}}]\label{hlocest}
For every subset $K\subset X$ and every $\ep>0$, we have
$$h(K,f)\le h(K,f,\ep)+h^*(f,\ep).$$
\end{proposition}

\dmph{
\begin{corollary}
If $(X,f)$ is asymptotically entropy expansive, then $h(f)<\infty$.
\end{corollary}
}

Asymptotic entropy expansiveness 
%is not a very strict condition and it 
holds for a  broad class of systems, as indicated by the following
proposition.

\begin{proposition}\label{aeesys}
%Let $(X,f)$ be asymptotically entropy expansive. Then the followings hold:
\begin{enumerate}

%\item The map $\mu\mapsto h_\mu(f)$ on $\cm(X,f)$ is upper semi-continuous.
%\item There is $\mu_{M}\in\cm_e(X,f)$ such that $h_{\mu_M}(f)=h(f)$.
%\item If $\Lambda$ is a compact $f$-invariant subset of $X$, then $(\Lambda,
%f|_\Lambda)$ is also asymptotically entropy expansive.
\item Every $C^\infty$ diffeomorphism on a compact manifold is asymptotically
entropy expansive %(cf. 
\cite[Theorem 2.2]{Buzzi}.
%).
\item Every $C^1$ diffeomorphism away from homoclinic tangencies is entropy
expansive %(cf. 
\cite[Theorem B]{LVY}.
\item If both $(X,f)$ and $(Y,g)$ are asymptotically entropy expansive, then
so is the product $(X\times Y,f\times g)$.
\item Every factor of an asymptotically entropy expansive system is asymptotically entropy expansive.
\end{enumerate}
\end{proposition}

%\begin{proposition}
%Let $(X,f)$ be asymptotically entropy expansive. Then the followings hold:
%\begin{enumerate}
%\item $h(f)<\infty$.
%\item The map $\mu\mapsto h_\mu(f)$ on $\cm(X,f)$ is upper semi-continuous.
%\item There is $\mu_{M}\in\cm_e(X,f)$ such that $h_{\mu_M}(f)=h(f)$.
%\item If $\Lambda$ is a compact $f$-invariant subset of $X$, then $(\Lambda,
%f|_\Lambda)$ is also asymptotically entropy expansive.
%\end{enumerate}
%\end{proposition}

\subsection{Invariant measures and metric entropy}\label{sectionmeasure}

Denote %by $\sB(X)$ the linear space of Borel measures on $X$,
by $\cM(X)$ the space of probability measures on $X$.
%, by $\cm(X,f)$ the subspace of invariant measures of 
%$(X,f)$ and by $\cm_e(X,f)$
%the subset of the ergodic ones. 
%Let $D$ be a fixed metric on the linear space of measures
%on $X$ that induces the weak-$*$ topology.
As $X$ is compact, both $\cM(X)$ and $\cm(X,f)$ are  compact metrizable spaces under the weak-$*$
topology \cite[Theorem 6.5 and Theorem 6.10]{Walters}. 
%We shall denote by $D^*$ the diameter of $\cm(X)$ under the metric
%specified in the following proposition.

\begin{proposition}[{\cite[Theorem 6.4]{Walters}}]
There is a metric $D$ on $\cM(X)$ such that $D$ induces the weak-$*$ topology on
$\cM(X)$ and
%the following properties hold:
%\begin{enumerate}
%\item $D(c\mu_1,c\mu_2)=cD(\mu_1,\mu_2)$ for any $c>0$ and any $\mu_1,\mu_2\in\cM(X)$.
%\item $D(\mu_0+\mu_1,\mu_0+\mu_2)=D(\mu_1,\mu_2)$ for any $\mu_0,\mu_1,\mu_2\in\cM(X)$.
%\item $D(\mu_1+\mu_2,\nu_1+\nu_2)\le D(\mu_1,\nu_1)+D(\mu_2,\nu_2)$ for any $\mu_1,\mu_2,\nu_1,\nu_2\in\cM(X)$.
%\end{enumerate}
$$D\left(\sum_{k=1}^na_k\mu_k,\sum_{k=1}^na_k\nu_k\right)
\le \sum_{k=1}^na_kD(\mu_k,\nu_k)$$ 
for any $n\in\ZZ^+$, any $\mu_1,\cdots,\mu_n,\nu_1,\cdots,\nu_n\in\cM(X)$ and any $a_1,\cdots
a_n>0$ satisfying $\sum\limits_{k=1}^na_k=1$.
\end{proposition}

\dmph{
\begin{proposition}[{\cite[Remarks in Section 6.1]{Walters}}]\label{weakconvopen}
A sequence $\mu_n\to\mu$  in $\cm(X,f)$, under the weak-$*$ topology,
if and only if
$$\liminf_{n\to\infty}\mu_n(U)\ge \mu(U)\text{ for every open subset $U\subset X$}.$$
\end{proposition}
}

Denote by $\ext{K}$ the set of extreme points of a convex set $K$.
By \cite[Theorem 6.10]{Walters}, $\cm_e(X,f)=\ext{\cm(X,f)}$
%$\cm_e(X,f)$ is exactly 
%the set of extreme points of $\cm(X,f)$ 
and 
$\cm(X,f)$ is a Choquet simplex, i.e. every $\mu\in\cm(X,f)$ 
is the barycenter of a unique probability measure supported on
$\ext{\cm(X,f)}$.
% the set of its extreme points (i.e. $\cm_e(X,f)$).
%As a corollary, 
%This implies that 
Moreover, $\cm_e(X,f)$ is a $G_\delta$ subset of $\cm(X,f)$.
If $\cm_e(X,f)$ is dense in $\cm(X,f)$, 
%is dense in $\cm(X,f)$. 
%in this case we know that 
then $\cm_e(X,f)$ is a residual subset of $\cm(X,f)$
and in this case $\cm(X,f)$ is a Poulsen simplex if it is %non-trivial, i.e. %it is 
not a singleton. 
%In the non-trivial case, the set of non-atomic ergodic measures is also residual
%\cite{Sig}.
The structure of the Poulsen simplex has been
studied in \cite{LOS}. Some important facts are listed below.
Readers are referred to \cite{Phelps} for more details on Choquet simplices.
%Hence we have the following corollary. 

\begin{proposition}[{\cite{LOS}}]\label{poulsen}
%Let $S$ be a Choquet simplex.
%and $S_0$ be the set of
%                extreme points of $S$.
\begin{enumerate}
\item A metrizable Choquet simplex $S$ is a Poulsen simplex if 
and only if $S$ is not a singleton and $\ext{S}$ is dense in $S$.
%$S=\overline{S_0}$,
%i.e. $S$ equals the closure of the set of
%                its extreme points.
\item The Poulsen simplex is unique up to affine homeomorphisms.
\item Suppose that $S$ is a Poulsen simplex. Then %$S_0$ 
$\ext{S}$ is homeomorphic
to the Hilbert space $\ell^2$.
%there is a homeomorphism
%$H:S\to[0,1]^{\ZZ^+}$ such that $H(S_0)=(-1,1)^{\ZZ^+}$
In particular, %$S_0$ 
$\ext{S}$ is arcwise connected by simple arcs.
\end{enumerate}
\end{proposition}

%By Proposition \ref{poulsen}, when $\cm(X,f)$ is a Poulsen simplex,

%We say that $(X,f)$ is \emph{uniquely ergodic} if $\cm(X,f)$ (equivalently, $\cm_e(X,f)$) is a singleton and 

%%%%%%%%%%%%%%%%
\iffalse

\begin{proposition}[{\cite{LOS}}]\label{poulsensimp}
Suppose that $\cm(X,f)$ is a Poulsen simplex. Then the followings hold.
\begin{enumerate}
\item The Poulsen simplex is unique up to affine homeomorphisms.
Hence
%we have that 
$\cm(X,f)$ is affinely homeomorphic to $\cm(\{0,1\}^{\ZZ},\sigma)$,
where $\sigma$ is the %two-sided 
full shift on $\{0,1\}^{\ZZ}$.
\item $\cm_e(X,f)$ is homeomorphic to the Hilbert space $\ell^2$.
\item $\cm_e(X,f)$ is arcwise connected by simple arcs.
\end{enumerate}
\end{proposition}

\begin{remark}
$(3)$ is a consequence of $(2)$ as
\begin{align*}
&D(\mu_1+\mu_2,\nu_1+\nu_2)\\\le{}&
D(\mu_1+\mu_2,\nu_1+\mu_2)+D(\nu_1+\mu_2,\nu_1+\nu_2)
\\={}&
D(\mu_1,\nu_1)+D(\mu_2,\nu_2).
\end{align*}
\end{remark}
\fi
%%%%%%%%%%%%%%%%%%%%%%%%%

\begin{definition}[{\cite[Theorem 1.1]{Ka80}}]\label{metent}
Let $\mu$ be an invariant probability measure for $(X,f)$. Fix $\delta\in(0,1)$. Denote
$$r_\mu(n,\ep,\delta):=\min
\left\{|\cU|:
\begin{gathered}\cU\text{ is a collections of $(n,\ep)$-balls}\\
\text{such that }
\mu\left(\bigcup_{U\in\cU}U\right)>1-\delta
\end{gathered}
\right\}.$$
Then the \emph{metric entropy} of $(X,f)$ with respect to $\mu$ 
%is given by
can be defined as
$$h_\mu(f):=\lim_{\ep\to0}\limsup_{n\to\infty}\frac{\ln r_\mu(n,\ep,\delta)}{n}
=\lim_{\ep\to0}\liminf_{n\to\infty}\frac{\ln r_\mu(n,\ep,\delta)}{n}.$$
\end{definition}

Throughout this article, 
by \emph{entropy map}
we mean
the map
$\mu\mapsto h_\mu(f)$ 
defined on $\cm(X,f)$.
%is referred to as 
%the \emph{entropy map}.
%This map is affine.

\begin{proposition}[{\cite[Theorem 8.1]{Walters}}]
For any $\mu,\nu\in\cm(X,f)$ and $\lambda\in[0,1]$, we have
$$h_{\lambda\mu+(1-\lambda)\nu}(f)=\lambda h_{\mu}(f)+(1-\lambda)h_\nu(f).$$
\end{proposition}

\begin{proposition}[Variational Principle]\label{propvp}
For any system $(X,f)$, we have
$$h(f)=\sup\bigl\{h_\mu(f):\mu\in\cM(X,f)\bigr\}
=\sup\bigl\{h_\mu(f):\mu\in\cM_e(X,f)\bigr\}.$$
\end{proposition}

%Throughout this article, the \emph{entropy map}
%is referred to the map
%$\mu\mapsto h_\mu(f)$ on $\cm(X,f)$.

\begin{proposition}[{\cite[Corollary 4.1]{Misiu}}]\label{uppersemicts}
If $(X,f)$ is asymptotically entropy expansive, then
the entropy map $\mu\mapsto h_\mu(f)$ is upper semi-continuous with
respect to the weak-$*$ topology on $\cM(X,f)$.
As a corollary, there is $\mu_{M}\in\cm_e(X,f)$, which is called a
measure of maximal entropy, such that $h_{\mu_M}(f)=h(f)$.
\end{proposition}

\subsection{The specification-like properties}
\begin{definition}\label{gapshadow}
%Denote 
%$$Z_0=\emptyset, Z_k:=\{1,2,\cdots,k\}\text{ for $k\in\ZZ^+$, $Z_\infty:=\ZZ$ and 
%$\bar\ZZ=\ZZ\cup\{\infty\}$}.$$
%For $k\in\bar\ZZ$, 
%By an \emph{orbit sequence} we mean a sequence $\cC$
Let $\sC=\{x_k\}
_{k\in\ZZ^+}$ be a sequence in $X$.
Let $\sss=\{m_k\}
_{k\in\ZZ^+}$
%of ordered pairs in $X\times\ZZ^+$
%an \emph{orbit sequence}. A \emph{gap} 
%for an orbit sequence
%can be associated with a \emph{gap} 
%of rank $k$ is a $(k-1)$-tuple
%is a sequence
 and $\sg=\{t_k\}_{k\in\ZZ^+}$
be sequences of positive integers.
The pair $(\sC,\sss)$ shall be called an
\emph{orbit sequence} while $\sg$ shall be called a \emph{gap}.
For $\ep>0$ and $z\in X$, we say that $(\sC,\sss, \sg)$ is \emph{$\ep$-traced} by $z$
%\in X$ 
if %the following \emph{shadowing property} holds:
for each $k\in\ZZ^+$,
\begin{equation}\label{eqtracing}
d(f^{s_{k}+j}(z), f^j(x_k))\le\ep\text{ for each }j=0,1,\cdots, m_k-1,
\end{equation}
where
$$s_1=s_1(\sss,\sg):=0\text{ and }s_k=s_k(\sss,\sg):=\sum_{i=1}^{k-1}(m_i+t_i-1)\text{ for }k\ge 2.$$
%\end{definition}
%\begin{definition}
%$(X,f)$ is said to have the weak specification property if for every $\ep>0$,
%there is $M>0$ such that any $(\sC,\sg)$ with $\min\sg\ge M$ can be $\ep$-shadowed.
\end{definition}

\begin{definition}\label{defspeci}
We say that $(X,f)$ has the \emph{specification property}
%, or called an \emph{SP system},
%or called a \emph{GO system},
 if for every $\ep>0$ there
is $M=M(\ep)>0$ such that for any orbit sequence 
$(\sC,\sss)$ and any gap $\sg$ satisfying $\min\sg\ge M$,
%such that %$\max\sg\le M(\ep)$ and 
there is $z\in X$ that $\ep$-traces $(\sC,\sss,\sg)$.
\end{definition}

\begin{definition}\label{defgo}
We say that 
$(X,f)$ %is said to  
has the \emph{gluing orbit property}
%or called a \emph{GO system},
 if for every $\ep>0$ there
is $M=M(\ep)>0$ such that for any orbit sequence 
$(\sC,\sss)$, there is a gap $\sg$ satisfying $\max\sg\le M$ and $z\in X$
such that %$\max\sg\le M(\ep)$ and 
$(\sC, \sss, \sg)$  can be $\ep$-traced by $z$.
\end{definition}

Let $\{a_n\}_{n=1}^\infty, \{b_n\}_{n=1}^\infty$ be two sequences of integers.
We write
$$\{a_n\}_{n=1}^\infty\le\{b_n\}_{n=1}^\infty\text{ if }a_n\le b_n\text{
for each }n\in\ZZ^+.$$
For a sequence $\sss=\{a_n\}_{n=1}^\infty$ of positive integers and a function $L:\ZZ^+\to\ZZ^+$,
%then 
we write
$$L(\sss):=\bigl\{L(a_n)\bigr\}_{n=1}^\infty.$$
We say that the function $L:\ZZ^+\to\ZZ^+$ is \emph{tempered}
if $L$ is nondecreasing and
$$\lim_{n\to\infty}\frac{L(n)}{n}=0.$$
Denote by $\sigma$ the shift operator on sequences, i.e.
$$\sigma\bigl(\{a_n\}_{n=1}^\infty\bigr)=\{a_{n+1}\}_{n=1}^\infty.$$

\begin{definition}\label{defawsp}
We say that $(X,f)$ %is said to  have 
has the
\emph{almost weak specification property} (as in \cite{GSW,QS}), %or called an \emph{MSP system},
or the \emph{tempered specification property}, %(the name we suggest to avoid ambiguity)
 if for every $\ep>0$ there
is a tempered function $L_\ep:\ZZ^+\to\ZZ^+$ such that for any orbit sequence 
$(\sC,\sss)$ and any gap $\sg$ satisfying $\sg\ge L_\ep(\sigma(\sss))$,
%such that %$\max\sg\le M(\ep)$ and 
there is $z\in X$ that $\ep$-traces $(\sC,\sss,\sg)$.
\end{definition}

\begin{definition}\label{defnugo}
We say that $(X,f)$ %is said to  have 
has the \emph{tempered gluing orbit property}
%or called an \emph{MGO system},
 if for every $\ep>0$ there
is a tempered function $L_\ep:\ZZ^+\to\ZZ^+$ such that for any orbit sequence 
$(\sC,\sss)$, there is a gap $\sg$ satisfying $\sg\le L_\ep(\sigma(\sss))$ and $z\in X$
such that %$\max\sg\le M(\ep)$ and 
$(\sC, \sss, \sg)$  can be $\ep$-traced by $z$.
\end{definition}

%\begin{definition}\label{perglu}%[cf. \cite{BV}\cite{}]
%$(X,f)$ is said to have \emph{periodic the gluing orbit property} 
%if for every $\ep>0$,
%there is $M(\ep)>0$ such that 
%for any orbit sequence $\sC$, there are $t\le M(\ep)$ and
%a gap $\sg$ with $\max\sg\le M(\ep)$ such that $(\sC,\sg)$ can be $\ep$-shadowed
%by a periodic point of the period $s_k+m_k+t$.
%\end{definition}

%\begin{remark}
Definition \ref{defspeci}--\ref{defnugo} are equivalent to their analogs
respectively, if we require that the tracing property
\eqref{eqtracing} holds for all finite orbit sequences.
A proof of the equivalence for the gluing orbit property can be found in \cite[Lemma 2.10]{Sun19}.
The proof for the other cases is analogous. 
The properties are called \emph{periodic} if for any finite orbit sequence we
require that
the tracing point $z$ is a periodic point with the specified period
(cf. \cite{KLO, TWW}).

%Note also that compared to the widely used
%definitions of specification and gluing
%orbit, we allow equalities in \eqref{eqshadow} for our convenience and this
%also has no influence.
%\end{remark}

The notion of the gluing orbit property first appeared in \cite{TS} by Tian
and W. Sun in an equivalent
form with the name ``transitive specification''. It has recently drawn much
attention since the work \cite{BTV} of Bomfim, Torres and Varandas. It is also
shown in \cite{TS} that the gluing orbit property holds if the system satisfies the
\emph{shadowing property} (also known as the \emph{pseudo-orbit tracing property})
and topological transitivity. Hence the result of \cite{LO} is obtained under a
stronger assumption than ours.

The notion of tempered specification property was first introduced,
without a name, in 
Marcus' remarkable work
\cite{Marc}
%by Marcus 
that proved this property for all ergodic toral automorphisms. 
The property is called \emph{almost weak specification} in some references such as
\cite{GSW, QS} and suggested to be called \emph{weak specification} in \cite{KLO}. The author suggests
the name \emph{tempered specification} to avoid %any 
possible ambiguity with other
specification-like properties. 
%Hence its variation should be called 
Then the tempered the gluing orbit is just a natural generalization
of the gluing orbit property and the tempered specification property.

The relations between various specification-like properties are illustrated
in Figure \ref{fig_summsb}.
Readers are referred to \cite{KLO} for a survey on specification-like properties.

\begin{figure}[h]%[5]{r}{3cm}
	
	\centering
	\begin{tikzpicture}%[Figure 1]
	[node distance=.8cm,
	start chain=going below,]

	\node (sp) [punktchain]  {Specification};

	\begin{scope}[start branch=venstre,
	%We need to redefine the join-style to have the -> turn out right
	every join/.style={->, ultra thick, shorten <=1pt}, ]
	
	\node[punktchain, on chain=going below, join=by {->}]
	(tsp) {Tempered Specification};
	
	\node[punktchain, on chain=going right, join=by {->}] (tgo) {Tempered Gluing Orbit};
	
	%\node[punktchain, on chain=going below, join=by {->}] (app) {Approximate Product\\ (APP)};
	
	\node[punktchain, on chain=going below right, join=by {->}, ] (app) {Approximate Product};

	%\node[punktchain, on chain=going left, join=by {<-}] (go)  {Gluing Orbit};
	
	%\node[punktchain, on chain=going above, join=by {<-}] (esp)      {Exact Specification};

	\end{scope}
	
	\begin{scope}[start branch=venstre,
	%We need to redefine the join-style to have the -> turn out right
	every join/.style={->, ultra thick, shorten <=1pt}, ]
	
	\node[punktchain, on chain=going right, join=by {->}] (go)  {Gluing Orbit};
	
	\node[punktchain, on chain=going above right, join=by {<-}] (st)      {Shadowing $+$ Transitivity};

	%\node[punktchain, on chain=going right, join=by {->}] (app)      {Approximate Product};

	%\node[punktchain, on chain=going above, join=by {<-}] (asp)      {Almost\\ Specification};
	
	\end{scope}

	%\node[punktchain, on chain=going below, join=by {->}](sapp){Strict Approximate\\ Product};

	\draw[->, ultra thick,] (go.south) -- (tgo.north);
	%\draw[dashed, ->, thick,] (go.east) -- (ct.west);

	%	\draw[dashed, teal, ultra thick, ] 
	%	($(sp.north)+(1.7, 0.4)$) rectangle ($(tsp.south)+(-1.7, -0.4)$); 
	
	%\node[teal] at (0.3,-3.9) {\bf mixing\\ positive entropy};
	%	\node[teal] at ($(tsp.south)+(0, -0.8)$) {\bf topologically mixing};
	%	\node[teal] at ($(tsp.south)+(0, -1.2)$) {\bf positive entropy};

	%	\draw[dashed, olive, ultra thick, ] %let \p1=(Linf.north), \p2=(w1.south) in
	%	($(go.north)+(1.7, 0.4)$) rectangle ($(tgo.south)+(-1.7, -0.4)$); %node[tubnode] {as $p$ decreases};
	
	%\node[teal] at (4.8,-3.9) {\bf transitive};
	%	\node[olive] at ($(tgo.south)+(0, -0.8)$) {\bf topologically transitive};
	
	%	\draw[dashed, blue, ultra thick, ] 
	%	($(sp.north)+(-2, 0.2)$) rectangle ($(st.south)+(1.8, -0.2)$); 
	
	%\node[blue] at (8,-1.3) {unique MME};
	%	\node[blue] at ($(go.east)+(1.7, -1.4)$) {\bf unique MME};
	%	\node[blue] at ($(go.east)+(2.2, -1.8)$) {\small [Climenhaga-Thompson 2016]};
	%	\node[blue] at ($(go.east)+(2, -2.4)$) {$\frac{L(n)}{ln n}\to 0$ [Pavlov 2019]};
	\end{tikzpicture}
	\caption{Relations between various specification-like properties}\label{fig_summsb}
\end{figure}

\subsection{The approximate product property}
%Denote 
%$$\ZZ_n:=\{0,1,\cdots, n-1\}\text{ and }\Sigma_n:=\{0,1,\cdots,n-1\}^{\ZZ^+}.$$
\begin{definition}\label{gapshadow}

Let $\sC=\{x_k\}
_{k\in\ZZ^+}$ be a sequence in $X$
 and $\sg=\{t_k\}_{k\in\ZZ^+}$
be an increasing sequence of nonnegative integers.
%For $n\in\ZZ^+$ and $\delta_1>0$,
%we say that $\sg$ is \emph{$(n,\delta_1)$-spaced}
%if
%$$t_1=0\text{ and }n\le t_{k+1}-t_k<n(1+\delta_1)\text{ for each }k\in\ZZ^+.$$
%The pair $(\sC,\sS)$ shall be called an
%\emph{orbit sequence} while $\sg$ shall be called a \emph{gap}.
%We denote
%$$L(\sg):=\max\{t_j:\}$$
%For $\ep>0$, $\delta_2>0$ 
For $n\in\ZZ^+$, $\delta_1,\delta_2,\ep>0$
and $z\in X$, we say that $\sC$ is 
\emph{$(n,\delta_1,\delta_2,\sg,\ep)$-traced} by $z$
%\in X$ 
if $\sg$ is \emph{$(n,\delta_1)$-spaced}, i.e.
$$t_1=0\text{ and }n\le t_{k+1}-t_k<n(1+\delta_1)\text{ for each }k\in\ZZ^+,$$
 and
the following \emph{tracing property} holds:
%For each $k\in\NN$,
\begin{equation}\label{ndeltrace}
\left|\left\{j\in\ZZ_n: d(f^{t_{k}+j}(z), f^j(x_k))>\ep\right\}\right|<\delta_2
n\text{ for each $k\in\ZZ^+$}.
\end{equation}
%where $t_0=0$.
%$$s_1=s_1(\sS,\sg):=0\text{ and }s_j=s_j(\sS,\sg):=\sum_{i=1}^{j-1}(m_i+t_i-1)\text{ for }j\ge 2.$$
%\end{definition}
%\begin{definition}
%$(X,f)$ is said to have the weak specification property if for every $\ep>0$,
%there is $M>0$ such that any $(\sC,\sg)$ with $\min\sg\ge M$ can be $\ep$-shadowed.
\end{definition}

\begin{definition}\label{defapp}
We say that $(X,f)$ %is said to have 
has the \emph{approximate product property},
%or called an \emph{APP system}, 
if for every $\delta_1, \delta_2, \ep>0$, 
there is $M=M(\delta_1,\delta_2,\ep)>0$
such that for every $n> M$ and every sequence $\sC$ in $X$, there are an $(n,\delta_1)$-spaced
sequence $\sg$ and $z\in X$ such that $\sC$ is \emph{$(n, \delta_1, \delta_2,
\sg,
\ep)$-traced} by $z$.
\end{definition}

%\begin{remark}
%Being an APP system 
The Approximate product property is almost the weakest specification-like property.
It is weaker than tempered gluing
orbit property and all other specification-like properties
discussed in \cite{KLO}, including almost specification property,
relative specification property, almost product property, etc.
%as well as . 
It is independent
with the decomposition introduced by Climenhaga
and Thompson \cite{CT}. 
%See Example \ref{nomaxent}.
%(cf. Example \ref{nomaxent}). 
%The properties that
%are stronger than the approximate product property and are not discussed 
%in this paper include

%\end{remark}

\begin{proposition}
Suppose that $(X,f)$ has the tempered gluing orbit property. 
Then $(X,f)$ has the approximate product property.
\end{proposition}

\begin{proof}
Suppose that we are given $\delta_1,\delta_2,\ep>0$ and $(X,f)$ has the tempered gluing orbit property.
There is a tempered function $L_\ep:\ZZ^+\to\ZZ^+$ such that for any orbit sequence 
$(\sC,\sss)$, there is a gap $\sg$ satisfying $\sg\le L_\ep(\sigma(\sss))$ and $z\in X$
such that %$\max\sg\le M(\ep)$ and 
$(\sC, \sss, \sg)$  can be $\ep$-traced by $z$. Then there is $M$ such that
$$\frac{L_\ep(n)}{n}<\delta_1\text{ for every }n> M.$$
For every $n> M$ and every sequence $\sC=\{x_k\}
_{k\in\ZZ^+}$ in $X$, assume that $(\sC,\{n\}^{\ZZ^+},\{t_k\}_{k=1}^\infty)$
is $\ep$-traced by $z$ and $t_k\le L_\ep(n)$ for each $k$. Denote
$$s_k:=\sum_{j=1}^{k-1}(n+t_k-1).$$
Then
$$s_1=0, n\le s_{k+1}-s_k=t_k-1<n+L_\ep(n)<(1+\delta_1) n\text{ for each $k$},$$
and
$$\left|\left\{j\in\ZZ_n: d(f^{s_{k}+j}(z), f^j(x_k))>\ep\right\}\right|=0<\delta_2
n\text{ for each $k$}.$$
Hence $\sC$ is $(n, \delta_1, \delta_2,
\{s_k\}_{k=1}^\infty, \ep)$-traced by $z$.
This implies that $(X,f)$ has the approximate product property.
\end{proof}

%We remark that the tempered gluing orbit property implies transitivity, while approximate
%product property does not. See Example \ref{exnontrans}.

%Analogous to \cite[Example 3.4]{BTV}, 
The following facts allow us to find more examples of systems with 
the approximate
product property.

\begin{proposition}[{\cite[Proposition 2.2]{PfSu}}]\label{appfactor}
	Every factor of a system with the approximate product property has approximate
	product property.
\end{proposition}

\begin{proposition}\label{prodapp}
Suppose that $(X,f)$ has the approximate product property and $(Y,g)$ has the tempered specification property.
Then the product $(X\times Y, f\times g)$ has the approximate product property.
\end{proposition}

\begin{proof}
Suppose that we are given $\delta_1,\delta_2,\ep>0$. 
We can find $\delta_1',\delta_2'>0$ such that
$$%\delta_1'<\frac{\delta_1}3, 
(1+\delta_1')^2
%(1+2\delta')
<1+\delta_1
%\delta_2'<\frac{\delta_2}{2}
\text{ and }\delta_2'(1+{\delta_1'})<\delta_2.$$
Let $M=M({\delta_1'},\delta_2',\ep)$ for $(X,f)$ as in Definition
\ref{defapp} and
 $L_{\ep}:\ZZ^+\to\ZZ^+$ be the tempered function
for $(Y,g)$ as in Definition \ref{defawsp}.
There is $N$ such that $L_{\ep}(n)<{\delta_1'}n$ for every $n>N$. Then for every
$n>\max\{M,N\}$ and every $\sC=\{(x_k,y_k)\}_{k=1}^\infty\in X\times Y$,
there are an $\Bigl(\bigl\lfloor(1+{\delta_1'})n\bigr\rfloor,{\delta_1'}\Bigr)$-spaced sequence $\sg=\{t_k\}_{k=1}^\infty$ and $z_X\in X$ such that
$\sC_X:=\{x_k\}_{k=1}^\infty$ is $\Bigl(\bigl\lfloor (1+{\delta_1'})n\bigr\rfloor,{\delta_1'},\delta_2',\sg,{\ep}\Bigr)$-traced by $z_X$.
%Let
%$$\sg':=\{t_k'\}_{k=1}^\infty \text{ such that }t_k':=t_{k+1}-t_k-n\text{
%for each }k.$$
Then for each $k$, we have
$$n\le\lfloor (1+{\delta_1'})n\rfloor\le t_{k+1}-t_k<(1+\delta_1')\lfloor (1+{\delta_1'})n\rfloor<(1+\delta_1)n$$
and
$$\left|\left\{j\in\ZZ_n: d(f^{t_{k}+j}(z_X), f^j(x_k))>\ep\right\}\right|<\delta_2'\lfloor (1+{\delta_1'})n\rfloor<\delta_2n.$$
This implies that $\sC_X$ is also $(n,\delta_1,\delta_2,\sg,{\ep})$-traced
by $z_X$.

Let
$$\sg^*:=\{t_k'\}_{k=1}^\infty \text{ such that }t_k':=t_{k+1}-t_k-(n-1)\text{
for each }k.$$
Then $t_k'\ge \delta_1'n> L_{\ep}(n)$ for each $k$ and hence
$\sg^*\ge L_{\ep}(\{n\}^{\ZZ^+})$. As $(Y,g)$ has tempered specification property, 
there is $z_Y\in Y$ that ${\ep}$-traces $(\{y_k\}_{k=1}^\infty,\{n\}^{\ZZ^+},\sg^*)$.
Then $\sC$ is $(n,\delta_1,\delta_2,\sg,\ep)$-traced by $(z_X,z_Y)$.

\end{proof}

In \cite{PfSu}, Pfister and Sullivan has shown that the approximate product property
implies entropy denseness. 
%if the entropy map is upper semi-continuous.

\begin{definition}[{\cite[Definition 2.7]{PfSu}}]
Let $\mu\in\cm(X,f)$. We say that $\mu$ is \emph{entropy-approachable}
by ergodic measures if for any $\eta>0$ and any $h<h_\mu(f)$, there is
$\nu\in\cm_e(X,f)$ such that
$$D(\mu,\nu)<\eta\text{ and }h_\nu(f)>h.$$
We say that the system $(X,f)$ is \emph{entropy-dense}
if every $\mu\in\cm(X,f)$ is entropy-approachable
by ergodic measures.
\end{definition}

\begin{proposition}[{\cite[Theorem 2.1]{PfSu}}]\label{entropydense}
Suppose that $(X,f)$ has the approximate product property. Then 
$(X,f)$ is \emph{entropy-dense}.
%, i.e.
%for every $\mu\in\cm(X,f)$,
%every $\eta,\gamma>0$, there is $\nu\in\cm_e(X,f)$ such that
%$$D(\mu,\nu)<\eta\text{ and }h_\nu(f)>h_\mu(f)-\gamma.$$
\end{proposition}

We remark that 
%compared with entropy-denseness, 
%of entropy-gene
almost entropy-approximability (see Definition \ref{defea})
implies entropy-approachability: by the Variational Principle,
the compact $f$-invariant set
$\Lambda$ supports an ergodic measure $\nu$ whose metric entropy
can be arbitrarily close to
$h(\Lambda,f)$, %\ge h(\Lambda, f,\ep_0)$
hence we have $h_\nu(f)>h$.
%The key point is that by 
The advantage of entropy-approximability is that
we have an upper estimate for the entropy of $h_{\nu}(f)$, from 
which we can derive entropy-genericity, %an even 
a more delicate structure
than entropy-denseness.

%Recall that $\cm(X,f)$ is a Choquet simplex. %(see e.g. \cite{Phelps}). Since
%Entropy denseness 
%The conclusion of Proposition \ref{entropydense} 
Entropy-denseness implies that $\cm_e(X,f)$
%, which is exactly the set of extreme points of $\cm(X,f)$, 
is dense in $\cm(X,f)$. By Proposition \ref{poulsen}, $\cm(X,f)$ is either a singleton or a Poulsen
simplex.
%in this case we know that 
%Then $\cm_e(X,f)$ is a residual subset of $\cm(X,f)$
%and $\cm(X,f)$ is a Poulsen simplex if it is non-trivial, i.e. %it is 
%not a singleton. 
%In the non-trivial case, the set of non-atomic ergodic measures is also residual.
%The structure of a Poulsen simplex has been
%studied in \cite{LOS}. Some important facts are list below.
%Hence we have the following corollary.

\begin{corollary}\label{apppoul}
Suppose that $(X,f)$ has the approximate product property. Then 
%the following hold.
%\begin{enumerate}
%\item 
$\cm_e(X,f)$ is a residual subset of $\cm(X,f)$
and
%\item If $(X,f)$ is not uniquely ergodic, then the set of non-atomic ergodic measures is a residual subset of $\cm(X,f)$. 
%\item $\cm_e(X,f)$ is homeomorphic to the Hilbert space $l^2$.
%\item 
$\cm_e(X,f)$ is arcwise connected by simple arcs.
%\end{enumerate}
\end{corollary}
%\end{corollary}

%Another remark is that only asymptotic entropy expansiveness and entropy denseness.
%(as defined in \cite{PfSu}), 
%even along with the fact that $\cm(X,f)$ is a Poulsen simplex,
%are not sufficient for the conclusions of Theorem \ref{thaea}
%and Theorem \ref{thmain}. See Example \ref{noappnothm}. 

Example \ref{noappnothm} provides an expansive system
that is entropy-dense
($\cm(X,f)$ is a Poulsen simplex) but not entropy-generic.

\section{Empirical Measures}\label{secem}

In this section we discuss some facts on empirical measures to prepare ourselves
for the proof of the main results. Our proof
%s of these lemmas 
mainly follows \cite[Section 5.3]{CLT}.
%We shall also need some lemmas %(Lemma \ref{nepcount} and \ref{qndelta}) 
%from the work of Pfister and Sullivan \cite{PfSu} for the purpose of entropy estimate. 
%We list them at the end of this section.

For $x\in X$ and $n\in\NN$, we define the \emph{empirical measure} $\cE(x,n)$ such that
$$\int\phi d\cE(x,n):=\frac1n\sum_{k=0}^{n-1}\phi(f^k(x))\text{ for every
}\phi\in C(X).$$
Given a set $U\subset\cm(X,f)$, let
$$E(U,n):=\{x\in X: \cE(x,n)\in U\}.$$
 Let $\mu\in\cm(X,f)$ and  $\eta>0$. Denote 
 $$\cB_\eta=\cB_\eta(\mu):=\overline{B(\mu,\eta)}
 =\left\{\nu\in\cm(X,f): D(\mu,\nu)\le\eta\right\}.$$ 
 For $N\in\ZZ^+$,
denote
\begin{align}
Z_{N,\eta}=Z_{N,\eta}(\mu):={}&\{x\in X:f^k(x)\in E(\cB_\eta,N)\text{ for every }k\in\NN\} \notag
\\\label{utrap}={}&\{x\in X:\cE(f^k(x),N)\in\cB_\eta\text{ for every }k\in\NN\}.
\end{align}
Then  $f(Z_{N,\eta})\subset Z_{N,\eta}$.
By \cite[Section 6.1]{Walters}, the map $x\mapsto \cE(x,N)$ is continuous. It is uniformly
continuous as $X$ is compact.
It follows that the set $Z_{N,\eta}$ is also compact.
%closed.
For $\ep>0$, denote
$$\var(\ep):=\max\Bigl\{D\big(\cE(x,1),\cE(y,1)\bigr):
d(x,y)\le\ep, x,y\in X\Bigr\}.$$
Uniform continuity of the map $x\mapsto \cE(x,1)$ implies that
%, we have
\begin{equation}\label{varep}
\lim_{\ep\to 0}\var(\ep)=0.
\end{equation}
We shall also denote by $D^*$ the diameter of $\cm(X)$, i.e.
$$D^*:=\max\bigl\{D(\mu,\nu): \mu,\nu\in\cm(X)\bigr\}.$$
%under the metric
%specified in the following proposition.

%\begin{lemma}
%For every $k$, we have $Z_{kn}\subset Z_n$.
%\end{lemma}

\begin{lemma}\label{measclose}
For any $N\in\ZZ^+$ and any $\nu\in\cm(Z_{N,\eta}, f)$, we have $D(\nu,\mu)\le\eta$.
\end{lemma}

\begin{proof}
Assume that $\nu\in\cm(Z_{N,\eta}, f)$ is ergodic. There is a generic point
$x\in Z_{N,\eta}$ such that $\cE(x,n)$ converges to  $\nu$ as $n\to\infty$.

Write $n=kN+l$ such that $k\in\NN$ and $0\le l<N$. Note that
\begin{align*}
\cE(x,n)&=\sum_{j=0}^{k-1}\left(\frac Nn\cE\left(f^{jN}(x),N\right)\right)
+\frac ln\cE\left(f^{kN}(x),l\right).
\end{align*}
For each $j\in\NN$, as $x\in Z_{N,\eta}$, by \eqref{utrap}, 
we have $\cE(f^{jN}(x),N)\in\cB_\eta$ and hence
$$D\left(\cE\left(f^{jN}(x),N\right),\mu\right)\le \eta.$$
It follows that
%Then %for $n>\frac{ND^*}{\eta}$, 
%We have
\begin{align*}
D\left(\cE(x,n),\mu\right)
%&=D\left(\sum_{j=0}^{k-1}\frac Nn\cE\left(f^{jN}(x),N\right)
%+\frac ln\cE\left(f^{kN}(x),l\right),\mu\right)\\
&\le\sum_{j=0}^{k-1}\frac NnD\left(\cE\left(f^{jN}(x),N\right),\mu\right)
+\frac ln D\left(\cE\left(f^{kN}(x),l\right),\mu\right)
\\&<%\frac{kN}{n}\cdot 
\eta+\frac {ND^*}n.
%\\&<6\eta.
\end{align*}
This implies that $D(\nu,\mu)\le \eta$ 
As $n\to\infty$, we have $\cE(X,n)\to\nu$ and hence .
%$\cE(x,n)\to\nu$.

When $\nu$ is not ergodic, the result follows from ergodic decomposition.
\end{proof}

\begin{lemma}\label{traceclose}
%Assume that
Let $\eta,\delta_1,\delta_2,\ep>0$ and $T, M\in\ZZ^+$ such that
\begin{equation}\label{tdelest}
\frac{2D^*}\eta<T\le\frac1{\delta_1}\text{ and }
\var(\ep)+(\delta_1+\delta_2)D^*<\eta.
\end{equation}
Suppose that $\sC$ is a sequence in $E(B(\mu,\eta),M)$ that is 
$(M,\delta_1,\delta_2,\sg,\ep)$-traced by $z$, where
$\sg$ is $(n,\delta_1)$-spaced.
Then $z\in Z_{TM,3\eta}$.
\end{lemma}

\begin{proof}
%Given $n\in\NN$, we need to show that $\cE(f^n(y),Tm)\in%\overline
%\cB$.
%Suppose that $y\in Y(\tau,\ep)$ and

%such that
%$(\sC_p,\sS_\tau,\sg)$ is $\ep$-shadowed by $y$.
%For each $k$, denote
%$$b_k:=\sum_{j=1}^{k}(\tau+t_j).$$
Given any $n\in\NN$, we need to show that 
$D\left(\cE(f^n(z),TM),\mu\right)<3\eta$. 

Denote $\sC=\{x_k\}_{k=1}^\infty$ and $\sg=\{t_k\}_{k=1}^\infty$. %Note that
There is unique $k$ such that $t_{k}<n\le t_{k+1}$.
Denote
$$s:=(t_{k+1}-n)+\sum_{j=1}^{T-2}(t_{k+j+1}-t_{k+j})=t_{k+T-1}-n.$$
%Then 
%$$(T-2)m\le s<(T-1)(m+NM).$$
%By \eqref{mmlimit}, we may assume that $m$ is large enough such that
%\begin{equation}\label{nmmt}
%NM<\frac{m}{T}.
%\end{equation}
By \eqref{tdelest},
we have
$$TM>(T-1)M(1+\delta_1)>s\ge (t_{k+1}-n)+(T-2)M.$$
We can write
\begin{align*}
\cE(f^n(z),TM)={}&
\frac{t_{k+1}-n}{TM}\cE(f^n(y),{t_{k+1}-n})\\&+
\sum_{j=1}^{T-2}\frac{t_{k+j+1}-t_{k+j}}{TM}\cE(f^{t_{k+j}}(y),t_{k+j+1}-t_{k+j})
\\&+\frac{TM-s}{TM}\cE(f^{t_{k+T-1}}(y),TM-s).
\end{align*}

For each $j$, denote
$$r_j:=\left|\Bigl\{l\in\ZZ_M: 
d\bigl(f^{t_{k+j}+l}(z), f^l(x_{k+j})\bigr)>\ep\Bigr\}\right|<\delta_2M.$$
As $\sC$ is 
$(M,\delta_1,\delta_2,\sg,\ep)$-traced by $z$, by \eqref{ndeltrace}
and \eqref{tdelest}, %and Lemma \ref{measclose},
%Then the tracing property \eqref{ndeltrace} yields that
we have
\begin{align*}
&D\left(\cE(f^{t_{k+j}}(z),t_{k+j+1}-t_{k+j}),\mu\right)\\
\le{}&\frac{M}{t_{k+j+1}-t_{k+j}}D\left(\cE(f^{t_{k+j}}(z),M),\mu\right)
\\&+\frac{1}{t_{k+j+1}-t_{k+j}}D\left(\cE(f^{t_{k+j}+M}(z),t_{k+j+1}-t_{k+j}-M),\mu\right)
\\<{}&%\frac{\tau}{\tau+t_{k+j+1}}\left(
D\left(\cE(f^{t_{k+j}}(z),M),\cE(x_{k+j},M)\right)
+D\left(\cE(x_{k+j},M),\mu\right)%\right)
\\&+\frac{(t_{k+j+1}-t_{k+j}-M)D^*}{t_{k+j+1}-t_{k+j}}
%+\delta_1D^{*}
\\<{}&\frac1M\left(\sum_{l=0}^{M-1}D\Bigl(\cE\left({f^{t_{k+j}+l}(z)},1\right),
\cE\left({f^{l}(x_{k+j})},1\right)\Bigr)\right)+\eta+\delta_1D^*
\\<{}&\left(\frac{M-r_j}{M}\var(\ep)+\frac{r_j}{M}D^*\right)+\eta+\delta_1D^*
\\<{}&2\eta.
\end{align*}
It follows that %for any $n\in\NN$,
\begin{align*}
D\left(\cE(f^n(z),TM),\mu\right)\le{}&
\frac{t_{k+1}-n}{TM}D^*+
%\sum_{j=1}^{T-2}\frac{\tau+t_{k+j+1}}{T\tau}D\left(\cE(f^{b_{k+j}}(y),\tau+t_{k+j+1}),\mu\right)
\frac{t_{k+T-1}-t_{k+1}}{TM}\cdot2\eta
%\\&
+\frac{TM-s}{TM}D^*
\\\le{}&\frac{2M}{TM}D^*+2\eta
\\<{}&3\eta.
\end{align*}
%Hence $\cE(f^n(y),Tm)\in\cB$ when $T$ is so large that
%\begin{equation}\label{rhoteta}
%\frac{2\rho}{T}<\eta.
%\end{equation}
\end{proof}

We shall also need 
the following facts %(Lemma \ref{nepcount} and \ref{qndelta}) 
from the work of Pfister and Sullivan 
\cite{PfSu} for %the purpose of 
entropy estimate.

\begin{definition}
Let $S$ be a subset of $X$.
For $n\in\ZZ^+$, $\delta>0$ and $\ep>0$, 
we say that $S$ is $(n,\delta,\ep)$-separated if for any distinct points
$x,y\in S$, we have
$$\left|\left\{k\in\ZZ_n: d(f^{k}(x), f^k(y))>\ep\right\}\right|>\delta n.$$
\end{definition}

By definition, if $0<\delta<\delta'$, then every
$(n,\delta',\ep)$-separated set is also $(n,\delta,\ep)$-separated.

%\begin{lemma}\label{shrinkdel}
%If $0<\delta<\delta'$, then every
%$(n,\delta',\ep)$-separated set is also $(n,\delta,\ep)$-separated.
%\end{lemma}

\begin{proposition}[{\cite[Proposition 2.1]{PfSu}}]\label{nepcount}
Let $(X,f)$ be any topological dynamical system.
Suppose that $\nu\in\cm_e(X,f)$ %be an ergodic measure 
and $h<h_\nu(f)$. Then there are $\delta>0$
and $\gamma>0$ such that for any
neighborhood $U$ of $\nu$, there is $N^*=N^*(h,\delta,\gamma,U)>0$ such that for any $n\ge N^*$ there
is an $(n,\delta, \gamma)$-separated set $\Gamma_n\subset E(U,n)$ with
$|\Gamma_n|\ge e^{nh}$.
\end{proposition}

\begin{lemma}[{%See \cite[1.3]{Ka80} or 
\cite[Lemma 2.1]{PfSu}}]\label{qndelta}
For $n\in\ZZ^+$ and $\delta\in(0,\frac12)$, denote 
$$Q(n,\delta):=|\{A\subset\ZZ_n:|A|\ge(1-\delta)n\}|.$$
Then
%$$%\lim_{n\to\infty}
\begin{equation}
\label{eqlnq}
\frac{\ln Q(n,\delta)}{n}\le-\delta\ln\delta-(1-\delta)\ln(1-\delta).%$$
\end{equation}
\end{lemma}

\begin{remark}
Note that
\begin{equation}\label{delest}
\lim_{\delta\to0}\left(-\delta\ln\delta-(1-\delta)\ln(1-\delta)\right)=0.
\end{equation}
\end{remark}

\section{Almost Entropy-Approximability}\label{ldsection}
In this section we prove Theorem \ref{thaea}. 
This is the crucial part of the article. Compared with the
result of Pfister and Sullivan \cite{PfSu},
%(which only considers the set $Z_{TM,3\eta}$),
%(Theorem \ref{entropydense}), 
we construct new
compact invariant sets and obtain fine  estimates of their entropies,
especially from the above.
%and establish
%an upper estimate of the entropy of $\nu$. 
This is carried out by combining
known techniques with
an argument originally developed in \cite{Sunze} by the author.

%Our proof  of Theorem \ref{alresidual} is a combination of the argument
%by Pfister and Sullivan
%\cite{PfSu} for entropy denseness and the argument 
%originally developed in \cite{Sunze}
%by the author that provides upper estimates of the entropy.

Theorem \ref{thaea} can be directly proved for any \emph{invariant}
measure $\mu$ by approximating it by a convex combination
of ergodic ones.
%with an argument similar to %the one in 
%\cite[Section 5.3]{CLT}.
Here we %can
take advantage of Proposition \ref{entropydense}
to make our exposition
more concise. 
So %what 
we just need 
to show that every \emph{ergodic} measure
is entropy-approximable.
%, as in the following proposition.

\begin{proposition}\label{prmain}
Let $(X,f)$ be a system with the approximate product property.
Suppose that $\mu_0\in\cm_e(X,f)$,
$h_0\in\bigl(0,h_{\mu_0}(f)\bigr)$ and
%any %$\eta, \ep%, \delta_0
%>0$
$\eta_0,\beta_0,\ep_0>0$. Then there are 
%$\delta\in(0,\delta_0)$ 
$\gamma\in(0,\ep_0)$
and a compact $f$-invariant subset 
$\Lambda=\Lambda(\mu_0, h_0, \eta_0, \beta_0, \gamma)$ such that
\begin{enumerate}%[(1)]
\item\label{measureclose}
$D(\nu,\mu_0)<\eta_0$ %\text{ 
 for every 
 $\nu\in\cm(\Lambda,f)$.
\item \label{mainhest}
$h(\Lambda,f)>h_0$ and $h(\Lambda,f,\gamma)<h_0+\beta_0$.
\end{enumerate}
\end{proposition}

\begin{proof}[Proof of Theorem \ref{thaea}]
Let $\mu\in\cm(X,f)$, $U$ be a neighborhood of of $\mu$,
$h\in(0,h_\mu(f))$ and $\ep,\beta>0$.
There is $\eta_0>0$ such that $B(\mu,2\eta_0)\subset U$.
By Proposition \ref{entropydense}, there is $\mu_0\in\cm_e(X,f)$
such that
$$D(\mu,\mu_0)<\eta_0\text{ and }h_{\mu_0}(f)>h.$$
By Proposition \ref{prmain}, there are $\gamma\in(0,\ep)$
and a compact $f$-invariant $\Lambda$ such that
$D(\nu,\mu_0)<\eta_0$ for every 
 $\nu\in\cm(\Lambda,f)$, $h(\Lambda,f)>h$ and $h(\Lambda,f,\gamma)<h+\beta$.
 It also follows that
$$\cm(\Lambda,f)\subset B(\mu,2\eta_0)\subset U.$$
\end{proof}

We shall prove Proposition \ref{prmain} in
Subsection \ref{mainconstruction}
and \ref{mainest}.
The proof is completed by Proposition \ref{cptinv} and \ref{entestlam}. 
In Subsection \ref{appminimal}
we discuss two corollaries concerning minimality.

\subsection{Construction}\label{mainconstruction}
Suppose that 
$(X,f)$ has the approximate product property
we are given 
$\mu_0\in\cm_e(X,f)$,
$h_0\in\left(0,h_{\mu_0}(f)\right)$ and
$\eta_0,\beta_0,\ep_0>0$.
%$\mu\in\cm_e(X,f)$, $\eta_0,\beta_0>0$ and $h_0\in(0,h_\mu(f))$.
%Let $\eta:=\dfrac{\eta_0}4$ and $\beta:=\dfrac{\beta_0}6$.
We fix
$$\text{$\eta:=\dfrac{\eta_0}4$, 
$\beta:=\frac1{20}\min\bigl\{{\beta_0},{h_{\mu_0}(f)-h_0},
h_0\bigr\}$, $h_1:=h_0+10\beta$}$$
and $T\in\ZZ^+$ such that
\begin{equation}\label{eqfix}
T\eta>2D^*.
\end{equation}
%%%%%%%%%
\iffalse
\begin{equation}\label{eqfix}
\text{$\eta:=\dfrac{\eta_0}4$, 
$\beta:=\frac1{20}\min\bigl\{{\beta_0},{h_\mu(f)-h_0},
h_0\bigr\}$, $h_1:=h_0+10\beta$ and $T>\dfrac{2D^*}\eta$}.
\end{equation}
\fi
%%%%%%%%%%%
% for $T\in\ZZ^+$
Note that $h_1+\beta<h_{\mu_0}(f)$. By Lemma \ref{nepcount}, there are $\delta_0>0$, $\gamma_0>0$ and
$N^*=N^*(h_1+\beta,\delta_0,\gamma_0,
B(\mu_0,\eta))$ such that for any $n\ge N^*$ there is an $(n,\delta_0,\gamma_0)$-separated
set $\Gamma_n^*\subset E(B(\mu_0,\eta),n)$ with
\begin{equation}\label{neplarge}
|\Gamma_n^*|>e^{n(h_1+\beta)}.
\end{equation}
% We may assume that $\delta_1<\frac1T$.
By \eqref{varep}, we  can fix $\ep>0$ such that
\begin{equation}\label{estep}
%h^*(f,2\ep)<\beta, 
\var(\ep)<\frac14\eta\text{ and }\ep<\frac13\min\{\ep_0,\gamma_0\}.
\end{equation}
%As $h(f)<\infty$, w
We %can 
fix $\delta_1>0$ such that
\begin{equation}\label{estdel1}
\delta_1<\min\left\{\frac1{T},\frac{\beta}{h_1},\frac{\beta}{\ln r(\ep)} \right\}.
%\text{ and } \delta_1(h(f)+\beta)<\beta.
\end{equation}
By \eqref{delest}, we can fix $\delta_2\in(0,\frac12)$ such that
\begin{equation}\label{estdel2}
\delta_2<\min\left\{\frac{\delta_0}2,\frac1{2T},\frac{\beta}{\ln r(\ep)}\right\}\text{ and }
0<-\delta_2\ln\delta_2-(1-\delta_2)\ln(1-\delta_2)<\beta.
\end{equation}
%%%%%%%%%%%%%
\iffalse
{\color{blue}
As 
$$\limsup_{n\to\infty}\frac{\ln r(n,\ep)}{n}=h_s(f,\ep)\le h(f),$$
there is $N_0>0$ such that
\begin{equation}\label{rnepest}
\ln r(n,\ep)< n(h(f)+\beta)\text{ for every }n\ge N_0.
\end{equation}
We fix $M>M(\ep,\delta_1,\delta_2)$ as in Definition \ref{defapp} such that 
\begin{equation}\label{mmest}
M>\max\left\{N^*,\frac{N_0}{\delta_1}\right\},
\;\frac{\ln(\delta_1M)}{M}<\beta
\text{
and }e^{M(h_1+\beta)}>e^{Mh_1}+1.
\end{equation}
%such that
%$$M>...$$
}
\fi
%%%%%%%%%%%%%%%%%%%%
Let $M(\ep,\delta_1,\delta_2)$ as in Definition \ref{defapp}.
% such that 
We fix $M\in\ZZ^+$ such that
\begin{equation}\label{mmest}
M>\max\left\{M(\ep,\delta_1,\delta_2),N^*
%,\frac{N_0}{\delta_1}
\right\},\;
%\text{ and } 
0<\frac{\ln(\delta_1M)}{M}<\beta
\text{
and }e^{M(h_1+\beta)}>e^{Mh_1}+1.
\end{equation}
By Lemma \ref{nepcount} and \eqref{neplarge}, there is an %fixed
 $(M,\delta_0,\gamma_0)$-separated set $\Gamma_M^*\subset E(B(\mu_0,\eta),M)$
with %\begin{equation}\label{neplarge}
$|\Gamma_M^*|>e^{M(h_1+\beta)}$.
%\end{equation}
We fix a subset $\Gamma_M\subset\Gamma_M^*$ such that
\begin{equation}\label{gammamest}
e^{Mh_1}\le|\Gamma_M|<e^{M(h_1+\beta)}.
\end{equation}
%We fix $\Gamma_M$.

Let 
\begin{equation}\label{eqm1}
M_1:=\lfloor\delta_1M\rfloor\le\delta_1M
\end{equation}
be the largest integer no more than
$\delta_1M$. 
%{\color{blue}Then $M_1\ge N_0$ by \eqref{mmest}.}
%where $\lfloor t\rfloor$ denotes the largest integer no more than $t\in\RR$.
Denote  $\Sigma:=\Sigma_{M_1}$ and $\Gamma:=(\Gamma_M)^{\ZZ^+}$.
For each $\xi=\bigl\{\xi(k)\bigr\}_{k=1}^\infty\in\Sigma$, denote
$$t_1(\xi):=0,\;
t_k(\xi):=\sum_{j=1}^{k-1}\bigl(M+\xi(j)\bigr)
\text{ for each }k\in\ZZ^+\text{ and }\sg_\xi:=\bigl\{t_k(\xi)\bigr\}_{k=1}^\infty.$$
For each $\xi\in\Sigma$ and each sequence 
$\sC=\bigl\{x_k(\sC)\bigr\}_{k=1}^\infty
\in (\Gamma_M)^{\ZZ^+}$,
denote
\begin{equation}\label{defycxi}
Y_{\sC,\xi}:=\bigl\{y\in X: 
\sC\text{ is $(M,\delta_1,\delta_2,\sg_\xi,\ep)$-traced by $y$}\bigr\}.
\end{equation}
Let
$$Y:=\bigcup_{\sC\in\Gamma,\xi\in\Sigma}Y_{\sC,\xi}.$$
Note that  by 
\eqref{eqfix},
\eqref{estep}, \eqref{estdel1} and \eqref{estdel2}, we have
$$\var(\ep)+(\delta_1+\delta_2)D^*<\frac14\eta+\frac{3D^*}{2T}<\eta.$$
Hence \eqref{tdelest} holds. By Lemma \ref{traceclose}, we have $Y\subset
Z_{TM,3\eta}$.

Denote by $\sigma_\Gamma$ and $\sigma_\Sigma$ the shift maps on $\Gamma$
and $\Sigma$, respectively. 
%It is clear that the following invariance property holds.

\begin{lemma}\label{shiftinv}
For every $\sC\in\Gamma$ and $\xi\in\Sigma$, we have
%If $(\sC_p, \sS_\tau, \{t_k\}_{k=1}^\infty)$ is $\ep$-shadowed by $x$, then
%$(\sC_p, \sS_\tau, \{t_{k+1}\}_{k=1}^\infty)$ can be $\ep$-shadowed by $f^{\tau+t_1}(x)$.
\begin{equation}\label{eqincl}
f^{t_2(\xi)}(Y_{\sC,\xi})\subset Y_{\sigma_\Gamma(\sC),\sigma_\Sigma(\xi)}.
\end{equation}
\end{lemma}

\begin{proof}
Take any $y\in Y_{\sC,\xi}$. Then $\sC$ is $(M,\delta_1,\delta_2,\sg_\xi,\ep)$-traced by $y$.
By Definition \ref{gapshadow} and \ref{defapp}, this implies that
$\sigma_{\Gamma}(\sC)$ is 
$(M,\delta_1,\delta_2,\sg_{\sigma_{\Sigma}(\xi)},\ep)$-traced by $f^{t_2(\xi)}(y)$.
It follows that
$f^{t_2(\xi)}(y)\in Y_{\sigma_\Gamma(\sC),\sigma_\Sigma(\xi)}$
and the inclusion \eqref{eqincl} holds.
\end{proof}

The following lemma shows that $Y$ is closed in $X$, hence is a compact set. 
%The argument
%takes advantage of our definition of the approximate product property
%,which allows equality in the tracing property \eqref{}.

\begin{lemma}
Let $\{y_n\}_{n=1}^\infty$ be a sequence in $Y$ such that $y_n\to \tilde y$ in $X$.
Then there
are $\tilde\sC\in\Gamma$ and $\tilde\xi\in\Sigma$ such that
%$G(x_n)\to\sg$ and 
%$(\sC_p, \sS_\tau, \sg)$ can be $\ep$-shadowed by $x$. 
$\tilde y\in Y_{\tilde\sC,\tilde\xi}$.
Hence $Y$ is compact.
\end{lemma}

\begin{proof}
%As $\{x_n\}$ is a Cauchy sequence and $G$ is uniformly continuous, 
%$\{G(x_n)\}$ is a Cauchy sequence in $\Sigma_M$. 
Denote
$$\Delta:=\{A\subset\ZZ_M:|A|\ge(1-\delta_2)M\}.$$
By Lemma \ref{qndelta}, $|\Delta|=Q(M,\delta_2)$ and
%. In particular,
$\Delta$ is finite.
For each $k\in\ZZ^+$ and $y\in Y_{\sC,\xi}$, denote
$$A_k(y):=\left\{\tau\in\ZZ_M: d(f^{t_k(\xi)+\tau}(y), f^\tau(x_k(\sC)))\le\ep\right\}\in
\Delta.$$
Assume that $y_n\in Y_{\sC_n,\xi_n}$
%$\ep$-shadows $(\sC_p,\sS_\tau,\sg(x_n))$ 
for each $n$.
Note that $\Gamma$, $\Sigma$ and $\Delta^{\ZZ^+}$ are compact metric %spaces as
symbolic spaces.
%By compactness of $\Sigma_M$, 
We can find a subsequence $\{n_j\}_{j=1}^\infty$, $\tilde\sC%=\{x_k\}_{k=1}^\infty
\in\Gamma$, $\tilde\xi\in\Sigma$ and
$\{A_k\}_{k=1}^\infty\in \Delta^{\ZZ^+}$
such that 
$$\sC_{n_j}\to\tilde\sC, \xi_{n_j}\to\tilde\xi
\text{ and }\bigl\{A_k(y_{n_j})\bigr\}_{k=1}^\infty\to\{A_k\}_{k=1}^\infty.$$
%There is $N_1$ such that for every $n_j>N_1$, we have
%$$t_l(\xi_{n_j})=t_l(\xi)\text{ for }$$
For each $k\in\ZZ^+$, %there is $\delta_k>0$ such that
%$$d(f^l(x),f^l(y))\le\frac\ep2\text{ for each $l=0,1,\cdots, \lceil kM(1+\delta_1)\rceil-1$, whenever }d(x,y)<\delta_k.$$
%As $y_n\to y$, there is $N_k$ such that
%$$d(y_n,y)<\delta_k\text{ for every }n>N_k.$$
%As $\xi_{n_j}\to\xi$, 
there is $N_k$ such that for every $n_j>N_k$,
we have
$$x_k(\sC_{n_j})=x_k(\tilde\sC),t_k(\xi_{n_j})=t_k(\tilde\xi)\text{ and }A_k(y_{n_j})=A_k
%\text{ for each }l=1,\cdots,k
.$$
%Then for any $n_j, n_{j'}>\max\{N_k, N_k'\}$, we asserts that
%$y_{n_j}, y_{n_{j'}}$ are not $(kM(1+\delta_1),\ep)$-separated. Hence
%we must have
%$$x_k(\sC_{n_j})=x_k(\sC_{n_{j'}})=:x_k$$ and
For each $\tau\in A_k$,  we have
\begin{align*}
d\Bigl(f^{t_k(\tilde\xi)+\tau}(\tilde y), f^\tau\bigl(x_k(\tilde\sC)\bigr)\Bigr)
&=\lim_{n_j\to\infty} d\Bigl(f^{t_k(\xi_{n_j})+\tau}(y_{n_j}), 
f^\tau\bigl(x_k(\sC_{n_j})\bigr)\Bigr)
%\\&=\lim_{k\to\infty} d(f^{s_j(k)+l}(x_{n_k}),f^l(p))\text{ (for $k>N$)}
%\\&
\le\ep.
\end{align*}
This implies that $\tilde y\in Y_{\tilde\sC,\tilde\xi}$.
%for $\tilde\sC:=\{x_k\}_{k=1}^\infty$.
%Hence $(\sC_p, \sS_\tau, \sg)$ can be $\ep$-shadowed by $x$.
\end{proof}

\begin{proposition}\label{cptinv}
%Let $(X,f)$ be a GO system with positive topological entropy. Then for
%every $h>0$, there is a compact invariant set $\Lambda$ such that for each 
%$n\in\ZZ^+$, we have
%$$s(\Lambda,n\tau,2\ep)\le(\tau+M)M^n|E_M|^n$$
%For $Y=Y(\tau,\ep)$, let
Let
\begin{equation}\label{eqlamdef}
\Lambda:=\bigcup_{k=0}^{M+M_1-1} f^k(Y).
\end{equation}
Then $\Lambda$ is a compact $f$-invariant subset of $Z_{TM,3\eta}$. 
In particular, $\Lambda$ verifies
Conclusion \eqref{measureclose} in Proposition \ref{prmain}.

%This implies that
%$$\text{$D\left(\cE(x,n),\mu\right)\le 3\eta<\eta_0$
%for every $x\in\Lambda$ and every $n>TM$},$$
%hence verifies Conclusion \eqref{measureclose} of Theorem \ref{localdense} for $N:=TM$.

%and for each 
%$n\in\ZZ^+$, we have
%\begin{equation}\label{entesti}
%s(\Lambda,n\tau,2\ep)%\le(\tau+M)C(n+2)|E_\ep|^{n+2}
%\le(\tau+M)M^{n+2}r(M-1,\ep)^{n+2}
%\end{equation}
%$$s(\Lambda,n\tau,2\ep)\le(\tau+M)M^{n+2}|E_\ep|^{n+2}$$
\end{proposition}

\begin{proof}
%Let
%$$\Lambda=\Lambda(\tau,\ep):=\bigcup_{k=0}^{\tau+M-1} f(Y(\tau,\ep)).$$
We have that $\Lambda$ is compact since $Y$ is compact.
We have $\Lambda\subset
Z_{TM,3\eta}$ since
$Y\subset
Z_{TM,3\eta}$
and $f(Z_{TM,3\eta})\subset Z_{TM,3\eta}$.

Now we need to show that $f(\Lambda)\subset\Lambda$.
Take any 
%For every 
$z\in\Lambda$.
By \eqref{eqlamdef}, there is $y\in Y$ and $\tau\in\{0,\cdots,M+M_1-1\}$ such that
$f^\tau(y)=z$. We have two cases to consider:
\begin{enumerate}[{Case }(1):]
\item\label{returnit1}
If $\tau<M+M_1-1$, then $f(z)=f^{\tau+1}(y)\in\Lambda$.
\item Suppose that $\tau=M+M_1-1$.
%There is $t\in\{\tau+1,\cdots, \tau+M\}$ such that
%Note that $\sg(y)_1\in
There are $\sC$, $\xi$ such that $y\in Y_{\sC,\xi}$. Note that
$M\le t_2(\xi)\le M+M_1$ and by Lemma \ref{shiftinv}, we have
$f^{t_2(\xi)}(y)\in Y$.
%Hence if $\tau=M+M_1-1$, 
%Hence
It follows from Case \eqref{returnit1} that
$$f(z)=f^{\tau+1}(y)=f^{M+M_1-t_2(\xi)}\left(f^{t_2(\xi)}(y)\right)\in
f^{M+M_1-t_2(\xi)}(Y)\subset\Lambda.$$
\end{enumerate}
%So we can conclude that $f(\Lambda)\subset\Lambda$.
As $\Lambda\subset Z_{TM,3\eta}$, by Lemma \ref{measclose},
we have
$$D(\nu,\mu_0)\le3\eta<\eta_0\text{ for every }\nu\in\cm(\Lambda, f).
%any $\nu\in\cm(Z_{N,\delta}, f|_{Z_{N,\delta}})$, we have $$.
$$

\end{proof}

\subsection{Entropy estimate}\label{mainest}
\begin{lemma}\label{nepsep}
Suppose that $y\in Y_{\sC,\xi}$ and $y'\in Y_{\sC',\xi'}$
such that
\begin{align*}
%t_k(\xi)=t_k(\xi')\text{ for }k=1,\cdots, n
%\\&x_k(\sC)=x_k(\sC')\text{ for }k=1,\cdots, n-1
t_n(\xi)=t_n(\xi')
\text{ and }
%\\&
x_n(\sC)\ne x_n(\sC').
\end{align*}
Then $y,y'$ are 
$(nM(1+\delta_1),\ep)$-separated. 
\end{lemma}

\begin{proof}
Denote $t:=t_n(\xi)=t_n(\xi')$. Denote
\begin{align*}
A:=&\left\{j\in\ZZ_M: d(f^{t+j}(y), f^j(x_n(\sC)))\le\ep\right\}\text{ and}
\\A':=&\left\{j\in\ZZ_M: d(f^{t+j}(y'), f^j(x_n(\sC')))\le\ep\right\}.
\end{align*}
By \eqref{defycxi} and %the tracing property
Definition \ref{gapshadow}, we have 
$$|A|,|A'|\ge(1-\delta_2)M.$$
It follows from \eqref{estdel2} that
$$|A\cap A'|\ge(1-2\delta_2)M>(1-\delta_0)M.$$
As $x_n(\sC)$ and $x_n(\sC')$ are distinct elements in $\Gamma_M$, they
are $(M,\delta_0,\gamma_0)$-separated. Then there must be $\tau\in A\cap
A'$ such that
$$d\Bigl(f^{\tau}\bigl(x_n(\sC)\bigr), f^{\tau}\bigl(x_n(\sC')\bigr)\Bigr)>\gamma_0>3\ep.$$
It follows that
\begin{align*}
d(f^{t+\tau}(y), f^{t+\tau}(y'))
\ge&d(f^{\tau}(x_n(\sC)), f^{\tau}(x_n(\sC')))
\\&-d(f^{t+\tau}(y), f^{\tau}(x_n(\sC)))
\\&-d(f^{t+\tau}(y'), f^{\tau}(x_n(\sC')))
\\>&\ep.
\end{align*}
Moreover, we have
$$t+\tau\le\sum_{k=1}^n\bigl(t_{k+1}(\xi)-t_k(\xi)\bigr)\le nM(1+\delta_1).$$
Hence $y,y'$ are 
$(nM(1+\delta_1),\ep)$-separated. 
\end{proof}

Denote by
$$C^\Gamma_{p_1\cdots p_n}=\{\sC\in\Gamma: x_k(\sC)=p_k\text{ for each }k=1,\cdots,n\}$$
a cylinder of rank $n$ in $\Gamma$ and
$$C^\Sigma_{w_1\cdots w_n}=\{\xi\in\Sigma: \xi(k)=w_k\text{ for each }k=1,\cdots,n\}$$
a cylinder of rank $n$ in $\Sigma$. 
Denote by $\cK^{\Gamma}_n$ and $\cK^{\Sigma}_n$
the %whole 
collections of all such cylinders, respectively. Denote
$$\cK^{\Gamma}:=\bigcup_{n=1}^\infty \cK^{\Gamma}_n
\text{ and }\cK^{\Sigma}:=\bigcup_{n=1}^\infty \cK^{\Sigma}_n.$$
For each cylinder $C^\Gamma\in \cK^{\Gamma}$ and $C^\Sigma\in \cK^{\Sigma}$,
%of the same rank
 denote
$$Y_{C^\Gamma,C^\Sigma}=\bigcup_{\sC\in C^\Gamma,\xi\in C^\Sigma}Y_{\sC,\xi}.$$
%The following is a direct corollary of Lemma \ref{nepsep}.
\begin{lemma}\label{cornep}
Suppose that $y_i\in Y_{C^\Gamma_i,C^\Sigma}$ for $i=1,2$ such that
$C^\Gamma_1,C^\Gamma_2$ are distinct cylinders
in $\cK^{\Gamma}_n$ and $C^\Sigma\in \cK^{\Sigma}_{n-1}$
%of rank $n$ in $\Gamma$ and $C^\Sigma$ is a cylinder of rank $n-1$ in $\Sigma$. 
Then $y_1,y_2$ are 
$(nM(1+\delta_1),\ep)$-separated. 

\end{lemma}

\begin{proof}
Assume that
$C^\Gamma_1=C^\Gamma_{p_1\cdots p_n}$
%\text{ and }
and $C^\Gamma_2=C^\Gamma_{q_1\cdots q_n}$. 
Let $n_0:=\min\{k:p_k\ne q_k\}$.
%Then $n_0\in\{1,2,\cdots,n\}$.
%As $C^\Gamma_1\ne C^\Gamma_2$, there is
%$n_0\in\{1,2,\cdots,n\}$ such that
There are 
$$\sC\in C^\Gamma_{p_1\cdots p_n}\subset C^{\Gamma}_{p_1\cdots p_{n_0}},
\;\sC'\in C^\Gamma_{q_1\cdots q_n}\subset C^{\Gamma}_{q_1\cdots q_{n_0}}$$
and $\xi,\xi'\in C^{\Sigma}$ such that
$y_1\in Y_{\sC,\xi}$, $y_2\in Y_{\sC',\xi'}$
%, $t_{n_0}(\xi)=t_{n_0}(\xi')$
and
$$x_{n_0}(\sC)=p_{n_0}\ne q_{n_0}=x_{n_0}(\sC').$$
As  $\xi,\xi'\in C^{\Sigma}\in \cK^{\Sigma}_{n-1}$
%We have $t_{n_0}(\xi)=t_{n_0}(\xi')$ since 
 and $n_0\le n$, we have
$$\text{$t_k(\xi)=t_k(\xi')$ for each $k=1,\cdots, n_0-1$}.$$
It follows that $t_{n_0}(\xi)=t_{n_0}(\xi')$.
%It follows from  
Hence by Lemma \ref{nepsep}, %that 
$y_1,y_2$ are 
$(nM(1+\delta_1),\ep)$-separated. 
\end{proof}

%Denote by $C(n)$ the number of different cylinders of rank $n$ in
%$\Sigma$, which is equal to
%the maximal cardinality of $(n,\frac13)$-separated subsets
%of $\Sigma$ for the shift map $\sigma$. Then
%$C(n)\le M^n$ and
%$$h(\Sigma,\sigma)=\limsup_{n\to\infty}\frac1n\ln C(n).$$

\begin{lemma}\label{nepestimate}
For   every $n\in\ZZ^+$, 
every cylinder $C^\Gamma=C^\Gamma_{p_1\cdots p_n}\in\cK^{\Gamma}_n$
and  every cylinder 
$C^\Sigma\in\cK^{\Sigma}_n$
%=C^\Sigma_{w_1\cdots w_{n}}$ in $\Sigma$% of rank $n$
, we have 
%there are at most
$$s\left(Y_{C^\Gamma,C^\Sigma},nM,2\ep\right)\le
\left(Q(M,\delta_2)r(\ep)^{\delta_2M}r(M_1,\ep)\right)^{n}.$$
%points in $Y_{C^\Gamma,C^\Sigma}$ that are $(nM,2\ep)$-separated.
\end{lemma}

\begin{proof}
Let $S(M_1,\ep)$ %=E(M(\ep),\ep)$ 
be a fixed $(M_1,
\ep)$-spanning subset of $X$ with the minimal cardinality.
Then $|S(M_1,\ep)|=r(M_1,\ep)$. Let $S_1$ be a fixed $(1,\ep)$-spanning subset
of $X$ with the minimal cardinality $r(\ep)$.
Let $\sA:=(A_1,\cdots, A_{n})$ be an $n$-tuple in $\Delta^{n}$.
%satisfying
%\begin{equation}\label{nofntuple}
%|A_k|>(1-\delta_2)M\text{ for each }k=1,2,\cdots, n.
%\end{equation}
We fix some $\xi\in C^{\Sigma}$. 
%Denote
%$$\sA:=\bigcup_{k=1}^n(t_k(\xi)+A_k).$$
Denote
\begin{align*}
Y_{C^\Gamma,C^\Sigma}(\sA):=\Bigl\{y\in Y_{C^\Gamma,C^\Sigma}:
{}&d\left(f^{t_{k}(\xi)+j}(y), f^j(p_k)\right)\le\ep\\&\text{
for every }j\in A_k, k=1,2,\cdots,n\Bigr\}
\end{align*}

Denote
$$\Omega(\sA):=\prod_{j=0}^{t_{n+1}(\xi)-1}\Omega_{j}(\sA),$$
where for $t_k(\xi)\le j<t_{k+1}(\xi)$, $k=1,2,\cdots, n$,
we put
\begin{align*}
\Omega_j(\sA):=\begin{cases}
\left\{f^{j-t_k(\xi)}(p_k)\right\}, &\text{ if } j-t_k(\xi)\in A_k;\\%\text{ for some }k;\\
S_1, &\text{ if }j-t_k(\xi)\in\ZZ_M\setminus A_k;\\%\text{ for some }k;\\
f^{j-t_k(\xi)-M}\bigl(S(M_1,\ep)\bigr), & \text{ if } t_k(\xi)+M\le j<t_{k+1}(\xi)\text{
for some }k.
%\text{ otherwise}.
\end{cases}
\end{align*}
Let $S_{\sA}$ be an $(t_{n+1}(\xi),2\ep)$-separated set in $Y_{C^\Gamma,C^\Sigma}(\sA)$.
For $y\in S_{\sA}$ and for each $j\in\ZZ_{t_{n+1}(\xi)}$,
%$j=0,1\cdots, t_{n+1}(\xi)-1$,
we can find $\pi_j(y)\in\Omega_j(\sA)$ such that
$d\bigl(f^j(y),\pi_j(y)\bigr)\le\ep$.
Let
$$\pi(y):=\left(\pi_1(y),\cdots,\pi_{t_{n+1}(\xi)-1}(y)\right)$$
The fact that $S_{\sA}$ is $(t_{n+1}(\xi),2\ep)$-separated
implies that $\pi:S_{\sA}\to\Omega(\sA)$
is an injection.
%Then there is an injection $\pi:S_{\sA}\to\Omega(\sA)$ such that
%for every $y\in S_{\sA}$, we have
%$$d(f^j(y),\pi_j(y))\le\ep\text{ for each }j=0,1\cdots, t_{n+1}(\xi)-1,$$
%where
%$$\pi(y)=\left(\pi_1(y),\cdots,\pi_{t_{n+1}(\xi)-1}(y)\right)
%\text{ such that }\pi_j(y)\in\Omega_{j}(\sA)\text{ for each }j.$$
%Hence
It follows that
$$|S_{\sA}|\le|\Omega(\sA)|\le
\left(|S_1|^{\delta_2M}|S(M_1,\ep)|\right)^{n}=
\left(r(\ep)^{\delta_2M}r(M_1,\ep)\right)^{n}.$$
%As $t_{k+1}(\xi)\le nM(1+\delta_1)$
%By Lemma \ref{qndelta}, there are at most $Q(M,\delta_2)^{n}$ $n$-tuples in $\Delta^{n}$.
%satisfying \eqref{nofntuple}. 
Note that
%We have
$$Y_{C^\Gamma,C^\Sigma}=\bigcup_{\sA\in\Delta^{n}}Y_{C^\Gamma,C^\Sigma}(\sA),$$
So
%By Lemma \ref{qndelta},
the maximal cardinality of  an $\left(t_{n+1}(\xi),2\ep\right)$-separated set in $Y_{C^\Gamma,C^\Sigma}$
is at most
$$%|\Delta^n|\left(r(\ep)^{\delta_2M}r(M_1,\ep)\right)^{n}\le
\sum_{\sA\in\Delta^{n}}|S_{\sA}|\le
\left(Q(M,\delta_2)r(\ep)^{\delta_2M}r(M_1,\ep)\right)^{n}.$$
The conclusion follows since $t_{n+1}(\xi)\ge nM$.
\end{proof}

\begin{lemma}\label{nepesty}
For  each 
%$\tau\in\ZZ_{M+M_1}$
%$\tau\in\{0,\cdots, M+M_1-1\}$ 
%and 
$n\in\ZZ^+$,
we have
$$s(\Lambda,nM,2\ep)<(M+M_1)\left(e^{M(h_1+\beta)}M_1Q(M,\delta_2)
r(\ep)^{\delta_2M}r(M_1,\ep)\right)^{n+2}.$$
%$$s(f^k(Y),n\tau,2\ep)\le C(n+2)|E_\ep|^{n+2}.$$
%$$s\bigl(f^\tau(Y), nM,2\ep\bigr)\le \left(e^{M(h_1+\beta)}M_1Q(M,\delta_2)
%r(\ep)^{\delta_2M}r(M_1,\ep)\right)^{n+2}$$ 
\end{lemma}

\begin{proof}
Given $n\in\ZZ^+$, 
we have
\begin{equation}\label{eqcountn}
|\cK^\Gamma_n|=|\Gamma_{M}|^n\text{ and }
|\cK^\Sigma_n|=M_1^n.
\end{equation}
It follows from \eqref{gammamest} and Lemma \ref{nepestimate} that
\begin{align*}
s(Y, nM,2\ep)&\le\sum_{C^\Gamma\in\cK^\Gamma_n,C^\Sigma\in\cK^\Sigma_n}
s\left(Y_{C^\Gamma,C^\Sigma},nM,2\ep\right)
\\&\le|\Gamma_{M}|^nM_1^n\left(Q(M,\delta_2)r(\ep)^{\delta_2M}r(M_1,\ep)\right)^{n}
\\&< \left(e^{M(h_1+\beta)}M_1Q(M,\delta_2)r(\ep)^{\delta_2M}r(M_1,\ep)\right)^{n}
\end{align*}
%The numbers of distinct cylinders of rank $n$ in
%$\Gamma$ and $\Sigma$ are $|\Gamma_{M}|^n$ and $M_1^n$, respectively.%
%Moreover, by \eqref{gammamest}, we have
%$|\Gamma_M|<e^{M(h_1+\beta)}$.
Note that if $S$ is an $(nM,2\ep)$-separated
subset of $f^\tau(Y)$, then $f^{-\tau}(S)$ includes an 
$(nM+\tau,2\ep)$-separated
subset of $Y$. 
For $\tau\in\ZZ_{M+M_1}$, this implies that
$$s\bigl(f^\tau(Y), nM,2\ep\bigr)\le
s\bigl(Y, nM+\tau,2\ep\bigr)
\le s\bigl(Y, (n+2)M,2\ep\bigr).$$
%Moreover, we have $$nM+\tau<(n+2)M$$ 
%The result follows since
%for $\tau\in\ZZ_{M+M_1}$.
%$\tau\in\{0,\cdots, M+M_1-1\}$.
%, we have
%$$t_{n+3}(\xi)\ge (n+2)M> nM+\tau.$$
It follows from  \eqref{eqlamdef} that
%\begin{align*}
%s\bigl(f^\tau(Y), nM,2\ep\bigr)&\le 
%s\bigl(Y, (n+2)M,2\ep\bigr)
%\\&\le\left(e^{M(h_1+\beta)}M_1Q(M,\delta_2)
%r(\ep)^{\delta_2M}r(M_1,\ep)\right)^{n+2}
%\end{align*}
\begin{align*}
s(\Lambda,nM,2\ep)&\le
\sum_{\tau=0}^{M+M_1-1}s\bigl(f^\tau(Y),nM,2\ep\bigr) %\notag
\\&\le
({M+M_1})s\bigl(Y,(n+2)M,2\ep\bigr) %\notag
\\ %\label{fewnep}
&<(M+M_1)\left(e^{M(h_1+\beta)}
M_1Q(M,\delta_2)r(\ep)^{\delta_2M}r(M_1,\ep)\right)^{n+2}.
\end{align*}
\end{proof}

\begin{proposition}\label{entestlam}
We have
$$%h_0< 
h(\Lambda, f)>h_0\text{ and }
h(\Lambda, f, 2\ep)< h_0+\beta_0.$$
Hence $\Lambda$ verifies Conclusion \eqref{mainhest} in
%supports an ergodic measure $\nu$ as requested in Proposition
Proposition \ref{prmain} for $\gamma:=2\ep\in(0,\ep_0)$.
\end{proposition}

\begin{proof}
Given $n\in\ZZ^+$, the approximate product property guarantees that
for each $C^\Gamma_n\in\cK^\Gamma_n$, there is some
$C^{\Sigma}_{n-1}\in\cK^{\Sigma}_{n-1}$ with $Y_{C^\Gamma,
C_*^\Sigma}\ne\emptyset$.
%Given $n\in\ZZ^+$, there are $|\Gamma_M|^n$ cylinders of rank $n$ in $\Gamma$
%and $M_1^{n-1}$ cylinders of rank $n-1$ in $\Sigma$. There must be a cylinder
%$C_*^\Sigma$ of rank $n-1$ such that
By \eqref{gammamest} and \eqref{eqcountn}, 
there is $C_*^\Sigma\in \cK^\Sigma_{n-1}$ such that
%$$|\cQ(C_*^\Sigma)|
$$\bigl|\{C^\Gamma\in \cK^\Gamma_n: Y_{C^\Gamma,
C_*^\Sigma}\ne\emptyset\}\bigr|\ge\frac{|\cK^\Gamma_n|}{|\cK^\Sigma_n|}=\frac{|\Gamma_M|^n}{M_1^{n-1}}
\ge\frac{e^{nMh_1}}{M_1^{n-1}},$$
%by \eqref{gammamest}, 
%where
%$$\cQ(C_*^\Sigma):=\{C^\Gamma: C^\Gamma\text{ is a cylinder of rank $n$ in $\Gamma$ and }Y_{C^\Gamma,
%C_*^\Sigma}\ne\emptyset\}.$$
%$$\cQ(C_*^\Sigma):=\{C^\Gamma\in \cK^\Gamma_n: Y_{C^\Gamma,
%C_*^\Sigma}\ne\emptyset\}.$$
By %Lemma \ref{nepsep}, 
Lemma \ref{cornep},
we have
\begin{align*}
s(\Lambda,nM(1+\delta_1),\ep)&\ge s(Y,nM(1+\delta_1),\ep)
%\\&\ge s(Y_{\Gamma, C_0^\Sigma},nM(1+\delta_1),\ep)
%\\&\ge|\cQ(C_*^\Sigma)|.
\\&\ge \bigl|\{C^\Gamma\in \cK^\Gamma_n: Y_{C^\Gamma,
C_*^\Sigma}\ne\emptyset\}\bigr|
\\&\ge\frac{e^{nMh_1}}{M_1^{n-1}}
\end{align*}
Hence by \eqref{estdel1}, \eqref{mmest} and \eqref{eqm1}, we have
\begin{align*}
h(\Lambda,f)&\ge
h(\Lambda,f,\ep)\\&\ge\limsup_{n\to\infty}\frac{\ln s(\Lambda, nM(1+\delta_1),\ep)}{nM(1+\delta_1)}
\\&\ge\limsup_{n\to\infty}\frac{nMh_1-(n-1)\ln M_1}{nM(1+\delta_1)}
\\&=\frac{h_1}{1+\delta_1}-\frac{\ln M_1}{M(1+\delta_1)}
\\&> h_1-\delta_1h_1-\frac{\ln(\delta_1M)}{M}
%\\&>h_1-\delta_1(h(f)+\beta)-\beta
\\&>h_1-2\beta
\\&> h_0.
\end{align*}

%By Lemma \ref{nepesty}, we have for each $n\in\ZZ^+$, 
%\begin{align}
%s(\Lambda,nM,2\ep)&\le
%\sum_{\tau=0}^{M+M_1-1}s(f^\tau(Y),nM,2\ep) \notag
%\\ \label{fewnep}&<(M+M_1)(e^{M(h_1+\beta)}M_1Q(M,\delta_2)r(\ep)^{\delta_2M}r(M_1,\ep))^{n+2}.
%\end{align}
%Then 
For every $t\in\ZZ^+$, there is $n=n(t)\in\ZZ^+$ such that $(n-1)M<t\le nM$.
%Hence 
Then it follows from Lemma \ref{nepesty} that
%, we have
\begin{align*}
h(\Lambda,f,2\ep)&=\limsup_{t\to\infty}\frac{\ln s(\Lambda, t,2\ep)}t 
\\&\le\limsup_{n\to\infty}\frac{\ln s(\Lambda, nM,2\ep)}{(n-1)M}
\\&\le(h_1+\beta)+\delta_2\ln
r(\ep)+
\frac{\ln M_1+\ln Q(M,\delta_2)+\ln r(M_1,\ep)}{M}.
\end{align*}
By \eqref{finent}, \eqref{eqlnq}, \eqref{estdel1}, \eqref{estdel2},
%\eqref{rnepest} 
\eqref{mmest}
and \eqref{eqm1}, we have 
\begin{align*}
&\delta_2\ln r(\ep)<\beta,\\
&\frac{\ln M_1}M\le\frac{\ln (\delta_1 M)}M<\beta,\\
%$$
&\frac{\ln Q(M,\delta_2)}{M}\le 
-\delta_2\ln\delta_2-(1-\delta_2)\ln(1-\delta_2)<\beta\text{ and}\\
%$$ $$
&\frac{\ln r(M_1,\ep)}{M}\le %\frac{\delta_1\ln r(M_1,\ep)}{M_1}
\frac{M_1\ln r(\ep)}{M}
\le\delta_1\ln r(\ep)<\beta.
%$$
\end{align*}
%Then by %\eqref{estep}, 
%\eqref{estdel1} and \eqref{estdel2}, we have
It follows that
$$%h(\Lambda, f,\gamma)\le
%h(\Lambda,f)\le 
h(\Lambda, f,2\ep)%+h^*(f,2\ep)
%<(h_1+\beta)+4\beta+\beta
<h_1+5\beta
<h_0+\beta_0.$$

%As $(X,f)$ is asymptotically entropy expansive, $\Lambda$ supports an ergodic measure $\nu$ with
%$h_\nu(f)=h(\Lambda,f)$. By Lemma \ref{measclose}, we have $D(\nu,\mu)\le 3\eta<\eta_0$.
\end{proof}

\subsection{Minimal systems}\label{appminimal}
%Existence of the set $\Lambda$ leads to t

In this subsection we prove Corollary \ref{cor_min},
which is split into
the following corollaries
%actually a corollary 
of Proposition \ref{prmain}.
Recall that $(X,f)$ is called \emph{minimal} if
$X$ has no nonempty proper compact and $f$-invariant subset.

%The following corollaries show that %for systems with 
%the approximate product property plus
%minimality implies zero topological entropy and unique ergodicity.

\begin{corollary}
Let $(X,f)$ be a system with the approximate
        product property and positive topological entropy. Then $(X,f)$ is
        not minimal.
\end{corollary}

\begin{proof}
Suppose that $h(f)>0$. By \eqref{eqhsup}, there is $\ep_0>0$ such that
$h(X,f,\ep_0)>0$. 
By the Variational Principle, there is an ergodic measure $\mu_0\in\cm_e(X,f)$
such that $h_{\mu_0}(f)>0$.
We can find $h_0\in\bigl(0,h_{\mu_0}(f)\bigr)$ and $\beta_0>0$ such that
$$0<h_0+\beta_0<h(X,f,\ep_0).$$
By Proposition \ref{prmain}, there are 
$\gamma\in(0,\ep_0)$
and a compact $f$-invariant subset 
$\Lambda$ such that
$$h(\Lambda, f, \gamma)<h_0+\beta_0<h(X,f,\ep_0)\le h(X,f,\gamma).$$
This implies that $\Lambda$ is a proper subset of $X$, hence $(X,f)$
is not minimal.
\end{proof}

\begin{corollary}
Let $(X,f)$ be a system with the approximate
        product property that is not uniquely ergodic. Then $(X,f)$ is
        not minimal.
\end{corollary}

\begin{proof}
Let $\mu_1,\mu_2$ be distinct ergodic measures and $0<\eta_0<\frac{1}{3}D(\mu_1,\mu_2)$.
%There 
By Proposition \ref{prmain}, there
are compact invariant sets $\Lambda_1$ and $\Lambda_2$ such that
$D(\nu,\mu_1)\le\eta_0$ for every $\nu\in\cm(\Lambda_1,f)$ %supported on $\Lambda_1$
and $D(\nu,\mu_2)\le\eta_0$ for every $\nu\in\cm(\Lambda_2,f)$. 
%supported on $\Lambda_2$.
It follows that $\Lambda_1\cap\Lambda_2=\emptyset$. This implies that $(X,f)$ is not minimal. 
\end{proof}

\section{Intermediate Entropies}\label{secintent}
%In this section, we prove Conclusion \eqref{thit1} and \eqref{thit2}
%of Theorem \ref{thmain}.
%\subsection{Entropy-approximability}

Theorem \ref{thmain} \eqref{thit1} is a corollary of Theorem \ref{thaea} and 
%the following
Proposition \ref{pr_ea}.

\begin{proposition}\label{pr_ea}
Suppose that $(X,f)$ is asymptotically entropy expansive system and
$\mu\in\cm(X,f)$ is almost entropy-approximable. 
Then $\mu$ is entropy-approximable.
\end{proposition}

\begin{proof}
Let %$\mu\in\cm(X,f)$, 
$U$ be a neighborhood of $\mu$,
$h\in(0,h_{\mu}(f))$ and $\beta>0$.
As $(X,f)$ is asymptotically entropy expansive, by Definition \ref{defexp},
there is $\ep>0$ such that 
$$h^*(f,\ep')<\frac{\beta}{2}\text{ for any }\ep'\in(0,\ep).$$
As $(X,f)$ is almost entropy-approximable, there are a compact $f$-invariant
set $\Lambda$ and $\gamma\in(0,\ep)$ such that
$$\cm(\Lambda, f)\subset U, \; h(\Lambda,f)>h\text{ and }
h(\Lambda, f, \gamma)<h+\frac{\beta}2.$$
by Proposition \ref{hlocest}, we have
% that
$$h(\Lambda, f)\le h(\Lambda, f, \gamma)+h^*(f,\gamma)<h+{\beta}.$$
%This implies that $(X,f)$ 
%It follows that
Hence $\mu$ is entropy-approximable.
\end{proof}

%\subsection{Entropy-Genericity}
%We prove our main results from the local denseness property \eqref{ldprop}.
%In this section we prove Theorem \ref{alresidual} and Theorem \ref{interent}.
%Note that if $h(f)=0$, then $\cm_e(X,f)$ is a residual subset of $\cm(X,f)$ by Proposition \ref{entropydense}
%and the theorems are trivial. In what follows, we assume that $h(f)>0$.

By Proposition \ref{uppersemicts}, asymptotic entropy expansiveness guarantees that the entropy map
is upper semi-continuous.
Hence Conclusion \eqref{thit2} of Theorem \ref{thmain} is a corollary of 
Conclusion \eqref{thit1} and 
%the following
Proposition \ref{pr_eg}.

\begin{proposition}\label{pr_eg}
%Let $(X,f)$ be an asymptotically entropy expansive system with approximate        product property.
%Suppose that \eqref{ldprop} holds.
Suppose that every $\mu\in\cm(X,f)$ is entropy-approximable
and the entropy map $\mu\mapsto h_\mu(f)$ is upper semi-continuous.
Then $(X,f)$ is entropy-generic.
\end{proposition}

\begin{proof}
%The case that $\al=0$

As the entropy map is upper semi-continuous,
we have that $\cm^{\alpha}(X,f)$ is a compact
metric subspace of $\cm(X,f)$, hence it is a Baire space.

For $0\le\al<\al'<h(f)$, denote 
$$\cm({\al,\al'}):=\{\mu\in\cm_e(X,f):\al\le h_\mu(f)<\alpha'\}%\text{ and }\cm_\al:=\cm_{\al,\infty}
.$$
By upper semi-continuity, $\cm(0,\al')$ is an open set. It follows that
 $$\cm(\al,\al')=\cm(0,\al')\cap\cm^{\al}(X,f)$$ is an open set in 
the subspace $\cm^{\al}(X,f)$.
Let
$$\cm_e(\al,\al'):=\cm({\al,\al'})\cap\cm_e(X,f)$$
As $\cm_e(X,f)$ is a $G_\delta$ set,
we have that
$\cm_e(\al,\al')$ is a $G_\delta$ set in the subspace
$\cm^{\al}(X,f)$.

%is  dense in $\cm^{\alpha}(X,f)$.

Suppose that we are given $\mu\in\cm^\al(X,f)$ and $\eta>0$.
By Proposition \ref{propvp}, upper semi-continuity of the entropy map guarantees that
on every compact $f$-invariant set $Y$ there is an ergodic
measure $\mu_Y$ such that $h_{\mu_Y}(f)=h(Y,f)$.
We fix $\mu_X\in\cm_e(X,f)$ with 
$$h_{\mu_X}(f)=h(f)>\al.$$
Recall that $D^*$ is the diameter of $\cm(X)$.
Denote
$$\mu':=\left(1-\frac{\eta}{3D^*}\right)\mu+\frac{\eta}{3D^*}\mu_X.$$
%-\mu),$$
%where
%$\mu_{M}\in\cm_e(X,f)$ such that $h_{\mu_M}(f)=h(f)$.
%As $h_{\mu_X}(f)=$, 
Then we have
$$D(\mu', \mu)<\frac\eta3\text{ and }h_{\mu'}(f)>\al.$$
%By Proposition \ref{entropydense}, 
As $\mu'$ is entropy-approximable,
there is a compact $f$-invariant
set $\Lambda$ such that
$$\cm(\Lambda, f)\subset B\left(\mu,\frac\eta3\right)\text{ and }
%$$
\al<h(\Lambda, f)<\min\bigl\{h_{\mu''}(f),\al'\bigr\}.$$
By Proposition \ref{propvp}, there is $\nu\in\cm_e(\Lambda,f)\subset B(\mu,\eta)$ such that
$$h_{\nu}=h(\Lambda,f)\in [\al,\al'),$$
It follows that $\nu\in\cm_e(\al,\al')$ and hence 
$$\cm_e(\al,\al')\cap %\bigl(
B(\mu,\eta)
%\cap \cm^\al(X,f)\bigr)
\ne\emptyset.$$
This %then 
implies that $\cm_e(\al,\al')$ is dense in $\cm^\al(X,f)$.

Consequently, each
$\cm_e(\al,\al')$ is residual in
$\cm^\al(X,f)$. %Finally, we have that
%It follows that
Hence
$$\cm_e(X,f,\al)=%(\cm_e(X,f)\cap\cm^\al)\cap(
\bigcap_{k=1}^\infty\cm_e({\al,\al+\frac1k})$$
is residual in $\cm^\al(X,f)$.
%If $\al=0$, %we just note that
%then $\cm^0(X,f)$ is exactly $\cm(X,f)$.
\end{proof}

%\begin{corollary}
%Let $(X,f)$ be an AEE and APP system. Then 
%\begin{enumerate}
%\item $H(X,f)=[0, h(f)]$.
%\item $\cm_z(X,f)$ is a residual subset of $\cm(X,f)$.
%\end{enumerate}
%\end{corollary}

%\begin{corollary}
%        Let $(X,f)$ be an asymptotically entropy expansive APP system.
%        For every $\mu\in\cm(X,f)$, every open neighborhood $U$ of $\mu$, denote
%        $$H(X,f,U):=\{h_\nu(f): \nu\in U\cap\cm_e(X,f)\}.$$
%        Then we have
%        $$\begin{cases}
%        H(X,f,U)\supset[0, h_\mu(f)], & \text{if }h_\mu(f)<h(f);\\
%        H(X,f,U)\supset[0, h_\mu(f)), & \text{if }h_\mu(f)=h(f).
%        \end{cases}$$
%        In particular, 
%        $H(X,f)=[0, h(f)]$.
%\end{corollary}

As a consequence of entropy-genericity,
Corollary \ref{cor_nbhd} \eqref{itnb1}
 follows from Theorem \ref{thmain} \eqref{thit2}
and Proposition \ref{localintent}.
%and the following proposition.
%The conclusion is just a consequence of entropy-genericity.

%As a consequence of entropy-genericity, there are ergodic measures
%of intermediate entropies in every neighborhood of an invariant measure.
%as shown in the following corollary.

\begin{proposition}\label{localintent}
%        Then we have
%We say that the system $(X,f)$ 
%\emph{has ergodic measures of intermediate entropies in every neighborhood} if
%For $U\subset\cm(X,f)$, denote
%$$\cH(X,f,U):=\bigl\{h_\nu(f): \nu\in U\cap\cm_e(X,f)\bigr\}.$$
If $(X,f)$ is entropy-generic, then
for every $\mu\in\cm(X,f)$ and every neighborhood $U$ of $\mu$, we have
        $$\begin{cases}
        \cH(X,f,U)\supset\bigl[0, h_\mu(f)\bigr], & \text{if }h_\mu(f)<h(f);\\
        \cH(X,f,U)\supset\bigl[0, h_\mu(f)\bigr), & \text{if }h_\mu(f)=h(f).
        \end{cases}$$
\end{proposition}

\begin{proof}%[Proof the Theorem \ref{interent}]
There is $\eta>0$ such that $B(\mu,2\eta)\subset U$.

Suppose that 
$$0\le\al<h_\mu(f)\le h(f).$$
Then
$\cm_e(X,f,\al)$ is residual in $\cm^\al(X,f)$,
%As $B(\mu,\eta)\cap\cm^\al(X,f)$ is a nonempty open set in
%$\cm^\al(X,f)$,
%By  Theorem \ref{alresidual}
%As $\cm_e(X,f,\al)$ is residual in $\cm^\al(X,f)$, 
hence it has nonempty intersection with the open subset
$B(\mu,\eta)\cap\cm^\al(X,f)$ in $\cm^\al(X,f)$, i.e.
%$$\cm_e(X,f,\al)\cap$$
there is an ergodic measure
$$\nu\in B(\mu,\eta)\cap\cm^\al(X,f)\subset U$$
such that $h_\nu(f)=\alpha$. 
%Hence we have $\al\in\cH(X,f,U)$.
It follows that 
\begin{equation}\label{eqlocalent}
\cH(X,f,U)\supset[0, h_\mu(f)).
\end{equation}

Suppose that $h_\mu(f)<h(f)$.
%, then
By Proposition \ref{propvp}, there is $\mu_0\in\cm(X,f)$ such that 
$h_{\mu_0}(f)>h_{\mu}(f)$. Denote
$$\mu':=\left(1-\frac{\eta}{D^*}\right)\mu+\frac{\eta}{D^*}\mu_0
\in
B(\mu,2\eta)\subset
U.$$
%\text{ and }h_{\tilde\mu}(f)>h_\mu(f).$$
We still have $h_{\mu'}(f)>h_\mu(f)$.
%and $\mu'\in U$.
%There is an ergodic measure
%$\nu\in B(\tilde\mu,\delta)\cap\cm^{h_\mu(f)}$ such that $h_\nu(f)=h_\mu(f)$.
%Then $\nu\in B(\mu,2\delta)\subset U$.
It follows 
from \eqref{eqlocalent}
that
$$\cH(X,f,U)\supset[0, h_{\mu'}(f))\supset[0, h_\mu(f)].$$
%In particular, by Proposition \ref{uppersemicts}, $(X,f)$ has ergodic measures
%of the maximal entropy.
\end{proof}

%%%%%%%%%%%%
\iffalse
By Theorem \ref{thmain} \eqref{thit2}, the conclusion of Proposition
\ref{localintent} holds for asymptotically entropy expansive systems
with the approximate product property.
We remark that it is possible that $h_\mu(f)\notin \cH(X,f,U)$ when $h_\mu(f)=h(f)$. This
happens if $(X,f)$ has multiple ergodic measures of maximal entropy and this
is compatible with the approximate product property. 
See Example \ref{nomaxent}.
It is clear that if in addition $(X,f)$ is intrinsically ergodic
(i.e. it has exactly one ergodic measure of maximal entropy) then
$$\cH(X,f,U)\supset[0, h_\mu(f)]\text{ for every }\mu\in\cm(X,f).$$
\fi
%%%%%%%%%%%%%%%%%%%%%%%%%%%%%%

%\section{Applications and Related Results}

%\subsection{Ergodic measures of intermediate entropies}
%Corollary \ref{interent} covers the following results:
%\cite{Sun12}, \cite[Section 3.2]{QS}, \cite{GSW}.
%\cite[Theorem 1.3]{Sunze}.
%Its continuous-time version also covers \cite[]{LSWW}

%\begin{itemize}
%\item 
%\end{itemize}

%\section{Lyapunov Exponents for Asymptotically Additive Potentials}\label{intexpo}

\section{Lyapunov Exponents and Pressures}
%\section{Asymptotically Additive Potentials}
\label{secaap}
The notion of asymptotically additive potentials
was introduced in \cite{FH}. 
\begin{definition}\label{defaapot}
A sequence $\Psi=\{\psi_n\}_{n=1}^\infty$ of continuous real-valued functions on $X$
is called a \emph{sub-additive potential} for the system $(X,f)$ if
for every $x\in X$ and every $m,n\in\ZZ^+$, we have
$$\psi_{m+n}(x)\le\psi_n(x)+\psi_m\bigl(f^n(x)\bigr).$$
%\text{ for every }x\in X, m,n\in\ZZ^+.$$
A sequence $\Phi=\{\phi_n\}_{n=1}^\infty$ of real-valued functions on $X$
is called an \emph{asymptotically sub-additive potential} for the system $(X,f)$, if for every $\eta>0$,
%there is a continuous function $\phi:X\to\RR$ such that
%$$\limsup_{n\to\infty}\frac1n\sup\{|\phi_n(x)-\sum_{k=0}^{n-1}\phi(f^k(x))|:x\in X\}<\eta.$$
there is a sub-additive potential $\Psi=\{\psi_n\}_{n=1}^\infty$ such that
$$\limsup_{n\to\infty}\frac1n\sup\Bigl\{\bigl|\phi_n(x)-\psi_n(x)\bigr|:
x\in X\Bigr\}<\eta.$$
We say that $\Phi$ is \emph{asymptotically additive} if both $\Phi$ and $-\Phi$
are asymptotically sub-additive.
\end{definition}

%From now on
Throughout this section, we assume that 
$\Phi=\{\phi_n\}_{n=1}^\infty$ is %assumed to be 
a fixed asymptotically additive potential for $(X,f)$.

\subsection{Intermediate Lyapunov exponents}\label{subile}
The \emph{Lyapunov exponent} for $\Phi$ with respect to an invariant measure
$\mu\in\cm(X,f)$ is defined as
$$\chi_\Phi(\mu):=\lim_{n\to\infty}\frac1n\int\phi_nd\mu.$$
%If $\mu$ is ergodic, then by Kingman's sub-additive ergodic theorem 
%(c.f. \cite[Theorem 10.1]{Walters}), we have that
%$$\frac{\phi_n(x)}{n}\to\chi_\Phi(\mu)\text{ as $n\to\infty$, 
%for $\mu$-almost every }x\in X.$$

\begin{proposition}[{\cite[Lemma A.4]{FH}}]\label{chicts}
%Let $\Phi$ be an asymptotically additive potential. 
%Let $\Phi$ be an asymptotically additive potential for $(X,f)$.
The map $\chi_\Phi:\cm(X,f)\to\RR$ is continuous. 
%with respect to the weak-$*$ topology.
\end{proposition}
Denote
$$\chi^\Phi_{\min}:=\inf\bigl\{\chi_\Phi(\mu):{\mu\in\cm(X,f)}\bigr\}\text{ and }
\chi^\Phi_{\max}:=\sup\bigl\{\chi_\Phi(\mu):{\mu\in\cm(X,f)}\bigr\}.$$
As $\cm(X,f)$ is compact, by Proposition \ref{chicts}, $\chi_\Phi$ attains its minimum and
maximum on $\cm(X,f)$. It is also clear that $\chi_\Phi$ is affine, hence the minimum and the maximum
can be obtained at extreme points of $\cm(X,f)$.%, i.e.
\begin{corollary}\label{corchimin}
There are
$\mu^\Phi_{\min},\mu^\Phi_{\max}\in\cm_e(X,f)$ such that
$$\chi_\Phi(\mu^\Phi_{\min})=\chi^\Phi_{\min}
\text{ and }\chi_\Phi(\mu^\Phi_{\max})=\chi^\Phi_{\max}.$$
\end{corollary}

Lyapunov exponents for asymptotically additive potentials were studied in \cite{FH}
and \cite{TWW}. In \cite{TWW}, it is shown that if $(X,f)$
has periodic gluing orbit property, then for each 
%$$a\in\cL_\Phi:=(\inf_{\mu\in\cm(X,f)}\chi_\Phi(\mu),\sup_{\mu\in\cm(X,f)}\chi_\Phi(\mu)),$$
$\al\in \left(\chi^\Phi_{\min},\chi^\Phi_{\max}\right)$, %where
%$$\chi^\Phi_{\min}:=\inf\{\chi_\Phi(\mu):{\mu\in\cm(X,f)}\}\text{ and }
%\chi^\Phi_{\max}:=\sup\{\chi_\Phi(\mu):{\mu\in\cm(X,f)}\},$$
%\sup_{\mu\in\cm(X,f)}\chi_\Phi(\mu)
there is an ergodic measure $\nu_\al$ of full support such that 
\begin{equation*}%\label{chiphinua}
\chi_\Phi(\nu_\al)=\al.
\end{equation*}
We say that $(X,f, \Phi)$ has the \emph{intermediate exponent property}
if for every $\al\in\left[\chi^\Phi_{\min},\chi^\Phi_{\max}\right]$,
there is $\nu_\al$ with $\chi_\Phi(\nu_\al)=\al$.
We realize that that this property follows from the denseness of 
$\cm_e(X,f)$. However, in general we do not know if
an ergodic measures of an intermediate exponent can have full support.

\begin{proposition}\label{pr_iep}
If $\cm_e(X,f)$ is dense in $\cm(X,f)$, then
$(X,f, \Phi)$ has the intermediate exponent property.
\end{proposition}

%We realized that denseness of $\cm_e(X,f)$ implies that
%there are ergodic measures of intermediate Lyapunov exponents:
\begin{proof}
%In this case 
By Proposition \ref{poulsen},
$\cm(X,f)$ is either a singleton or a Poulsen
simplex. We must have that 
%By Proposition \ref{poulsen}, 
$\cm_e(X,f)$ is arcwise
connected. Hence there must be an arc in $\cm_e(X,f)$ that connects $\mu^\Phi_{\min}$ and $\mu^\Phi_{\max}$,
on which 
for every $\al\in\left[\chi^\Phi_{\min},\chi^\Phi_{\max}\right]$
there is an ergodic measure $\nu_\al$ with $\chi_\Phi(\nu_\al)=\al$.
\end{proof}

Corollary \ref{intexpon} follows from 
Corollary \ref{apppoul} and
Proposition \ref{pr_iep}.

\begin{corollary}\label{intexpon} 
%Suppose that $\cm_e(X,f)$ is dense in $\cm(X,f)$. 
%Let $(X,f)$ be a %n asymptotically entropy expansive 
%system with 
Suppose that $(X,f)$ has the
approximate product property and 
$\Phi$ is an asymptotically additive potential for $(X,f)$. 
Then $(X,f, \Phi)$ 
%$(X,f)$ 
has the intermediate exponent property.
% for asymptotically additive potentials.
%for every $a\in[\chi^\Phi_{\min},\chi^\Phi_{\max}]$, there is $\nu_a\in\cm_e(X,f)$ such that
%$\chi_\Phi(\nu_a)=a$.
\end{corollary}

\subsection{Pressure-Genericity}\label{subpregen}
%\subsection{Invariant sets with intermediate pressures}
%Let $\Phi$ be an asymptotically additive potential for $(X,f)$.

%Define
For $n\in\ZZ^+$ and $\ep>0$, define
$$P(X, f,\Phi,n,\ep):=\sup\left\{\sum_{x\in S}e^{\phi_n(x)}: 
S\text{ is an $(n,\ep)$-separated
subset of }X\right\}$$
and
$$P(X, f,\Phi,\ep):=\limsup_{n\to\infty}\frac{\ln P(X, f,\Phi,n,\ep)}n.$$
The \emph{topological pressure} of $(X, f,\Phi)$ is given by
$$P(f,\Phi)=P(X,f,\Phi):=\lim_{\ep\to 0}P(X, f,\Phi,\ep)
=\sup_{\ep> 0}P(X, f,\Phi,\ep).$$
%\limsup_{n\to\infty}\frac{\ln P(f,\Phi,n,\ep)}n.$$
For each $\mu\in\cm(X,f)$, the \emph{pressure} of $\Phi$ with respect to $\mu$
is defined as
$$P_\Phi(\mu):=h_\mu(f)+\chi_\Phi(\mu).$$
We have the following Variational Principle.
% from \cite{CFH} and \cite{FH}.
%Note that in
%the additive case we have $\chi_\Phi(\mu)>-\infty$ for every $\mu\in\cm(X,f)$.

\begin{proposition}[{\cite[Theorem 3.1]{FH}}]\label{ppvp}
It holds that
$$P(f,\Phi)=\sup\bigl\{P_\Phi(\mu):\mu\in\cm(X,f)\bigr\}
=\sup\bigl\{P_\Phi(\mu):\mu\in\cm_e(X,f)\bigr\}.$$
\end{proposition}

As a generalization
of the classical pressure for a continuous potential (cf. \cite[Chapter 9]{Walters}),
the pressure for an asymptotically additive potential 
shares similar properties.
It is clear that  $P_\Phi(\cdot)$ is affine on $\cm(X,f)$.
By Proposition \ref{chicts}, if the entropy map is 
upper semi-continuous, then so is
$P_{\Phi}(\cdot)$.
\dmph{
and there is $\mu^P_X
\in\cm_e(X,f)$ such that 
\begin{equation}\label{pmax}
P_\Phi\left(\mu^P_X\right)=P(f,\Phi).
\end{equation}
}

By combining the argument in Section \ref{ldsection} and the argument in
\cite[Section 3]{Sunct},
one can directly prove an analog of Proposition \ref{prmain} for pressures.
Here we just show that
%the following %proposition as a corollary 
%fact 
it is implied by
%a consequence of 
entropy-approximability.

\begin{proposition}\label{pr_prap}
%Let $(X,f)$ be an asymptotically entropy expansive system
%with the approximate product property
%and $\Phi$ be an asymptotically additive potential.
Suppose that $\mu\in\cm(X,f)$ is
entropy-approximable, $U$ is a neighborhood of $\mu$,
$\al\in\bigl(\chi_{\Phi}(\mu),P_{\Phi}(\mu)\bigr)$ and
%any %$\eta, \ep%, \delta_0
%>0$
%$\eta,
$\beta>0$. Then there 
is
%are %$\delta\in(0,\delta_0)$ 
%$\gamma\in(0,\ep)$ and 
a compact $f$-invariant subset 
$\Lambda$
%=\Lambda(\mu_0, h_0, \eta_0, \beta_0, \gamma)$ 
such that
$$\cm(\Lambda, f)\subset U\text{ and }
%$$
\al<P(\Lambda, f, \Phi)<\al+\beta.$$
\end{proposition}

\begin{proof}
We fix %$\ep>0$ such that
$$\ep:=\frac13\min\{\al-\chi_\Phi(\mu), P_\Phi(\mu)-\al, \beta\}>0.$$
It follows that
$$0<\al-\chi_\Phi(\mu)+\ep<
P_\Phi(\mu)-3\ep-\chi_\Phi(\mu)+\ep<
P_\Phi(\mu)-\chi_\Phi(\mu)=h_\mu(f).$$
%As $\chi_{\Phi}$ is continuous
By Proposition \ref{chicts},, there is a neighborhood $U_0$
of $\mu$ such that
$$\left|\chi_\Phi(\nu)-\chi_\Phi(\mu)\right|<\ep
\text{ for every }\nu\in U_0.$$
%B(\mu,\eta_0).$$
%There is $\eta\in(0,\eta_0)$ such that $B(\mu, \eta)\subset U$.
As $\mu$ is entropy-approximable, there is a compact $f$-invariant
set $\Lambda$ such that
$$\cm(\Lambda, f)\subset U\cap U_0\subset U$$
%B(\mu,\eta)
%\text{ and }
and
$$
\al-\chi_\Phi(\mu)+\ep<h(\Lambda, f)<\al-\chi_\Phi(\mu)+2\ep.$$
Then by Proposition \ref{ppvp}, we have
\begin{align*}
P(\Lambda, f, \Phi)
&=\sup\bigl\{h_\nu(f)+\chi_\Phi(\nu):\nu\in\cm(\Lambda,f)\bigr\}
\\&\ge\sup\bigl\{h_\nu(f):\nu\in\cm(\Lambda,f)\bigr\}
+\inf\bigl\{\chi_\Phi(\nu):\nu\in\cm(\Lambda,f)\bigr\}
\\&> h(\Lambda,f)+\bigl(\chi_\Phi(\mu)-\ep\bigr)
\\&>\al
\end{align*}
and
\begin{align*}
P(\Lambda, f, \Phi)
&=\sup\bigl\{h_\nu(f)+\chi_\Phi(\nu):\nu\in\cm(\Lambda,f)\bigr\}
\\&\le\sup\bigl\{h_\nu(f):\nu\in\cm(\Lambda,f)\bigr\}
+\sup\bigl\{\chi_\Phi(\nu):\nu\in\cm(\Lambda,f)\bigr\}
\\&< h(\Lambda,f)+\bigl(\chi_\Phi(\mu)+\ep\bigr)
\\&<\al+3\ep
\\&=\al+\beta.
\end{align*}
\end{proof}

%Conclusion \eqref{thit3}
%of Theorem \ref{thmain} can be shown either as a corollary of

%Pressure-genericity 
Conclusion \eqref{thit3}
of Theorem \ref{thmain}
can be proved in two ways:
%\begin{enumerate}
either as a consequence of entropy-approximability and
Proposition \ref{pr_prap}, with
an argument analogous to the proof of Proposition \ref{pr_eg};
or as a consequence of entropy-genericity and
Proposition \ref{presgen}. We shall present the latter.

\begin{lemma}\label{presdense}
Suppose that $(X,f)$ is entropy-generic,
%the entropy map is upper semi-continuous 
%and 
$h(f)>0$ and
%Assume that 
$$P_{\inf}(f,\Phi)<\al<P(f,\Phi).$$ Then
for any $\al'>\al$, the set
$$\spp_e(\al,\al'):=\{\mu\in\cm_e(X,f):\chi_\Phi(\mu)\le\al\le P_\Phi(\mu)<\al'\}$$
is dense in $\spp^\al(X,f,\Phi)$.
%$\spp^\al$.        
\end{lemma}

\begin{proof}
Recall that
$$\spp^\al(X,f,\Phi)=\{\mu\in\cm(X,f):\chi_\Phi(\mu)\le\al\le P_\Phi(\mu)\}.$$
Let $\mu_0\in\spp^\al(X,f,\Phi)$ and $\eta_0>0$. We need to show that
%there is 
$$\spp_e(\al,\al')\cap B(\mu_0,\eta_0)\ne\emptyset.$$
%\spp_e(\al,\al')$ such that $D(\nu,\mu_0)<\delta_0$.

%As $P_{\inf}(f,\Phi)<\al<\al'<P(f,\Phi)$, 
By Proposition \ref{ppvp}, 
we can fix $\mu_M,\mu_m\in\cm(X,f)$
such that 
%%$$P_\Phi(\mu_M)>\al'\text{ and }
$$P_\Phi(\mu_m)<\al<P_\Phi(\mu_M).$$

Our discussion splits into the following cases: 

\begin{enumerate}[\bf {Case }1.]
\item\label{c1} 
%\textbf{\flushleft Case 1.} 
Suppose that $\chi_\Phi(\mu_0)<\al<P_\Phi(\mu_0)$.

We fix 
$$\ep:=\dfrac12\min\bigl\{\al-\chi_\Phi(\mu_0),
\al'-\al, P_\Phi(\mu_0)-\al\bigr\}>0.$$
By Proposition \ref{chicts},
there is $\eta_1\in(0,\eta_0)$ such that
$$|\chi_\Phi(\mu)-\chi_\Phi(\mu_0)|<\ep
\text{ for every }\mu\in B(\mu_0,\eta_1).$$
%Note that 
%$$h_{\mu_0}(f)=P_\Phi(\mu_0)-\chi_\Phi(\mu_0).$$
%By Theorem \ref{interent}
As $(X,f)$ is entropy-generic, there is an ergodic measure 
$\nu\in B(\mu_0,\eta_1)$
such that
$$h_{\nu}(f)=\al-\chi_\Phi(\mu_0)+\ep\in\bigl(0,h_{\mu_0}(f)\bigr).$$
%Then
%$$\chi_\Phi(\nu)<\chi_\Phi(\mu_0)+\eta<\al,$$
It follows that
$$\chi_\Phi(\nu)<\chi_\Phi(\mu_0)+\ep<\al$$
and
$$\al<P_\Phi(\nu)=\chi_\Phi(\nu)+
\bigl(\al-\chi_\Phi(\mu_0)+\ep\bigr)
<\al'.
%\min\{P_\Phi(\mu_0),\al'\}-2\eta
%\in%(\al,\al+2\ep)\subset
%(\al,\al')
$$
%and
%$$P_\Phi(\nu)=\chi_\Phi(\nu)+\min\{P_\Phi(\mu_0),\al'\}-\chi_\Phi(\mu_0)-\eta<\min\{P_\Phi(\mu_0),\al'\}<\al'$$
Hence we have $\nu\in\spp_e(\al,\al')\cap B(\mu_0,\eta_0)$.

\item\label{c2}
%\textbf{\flushleft Case 2.} 
Suppose that  $\chi_\Phi(\mu_0)=\al<P_\Phi(\mu_0)$.

%As $\al>P_{\inf}(f,\Phi)$,
%there is $\mu'\in\cm(X,f)$ such that
%Note that 
%\begin{equation}\label{algreat}
%\al>P_\Phi(\mu')\ge\chi_\Phi(\mu').
%\end{equation}
%Let  $$\eta=\dfrac13\min\{\al'-\al, P(\mu_0)-\al\}.$$
There is $\eta_1\in\left(0,\dfrac{\eta_0}2\right)$ such that
$$P_\Phi(\mu_1)>\al\text{ for }
\mu_1:=\left(1-\frac{\eta_1}{D^*}\right)\mu_0+\frac{\eta_1}{D^*}\mu_m.$$
As $\al>P_\Phi(\mu_m)\ge\chi_\Phi(\mu_m)$, we have
$$\mu_1\in B\left(\mu_0,\frac{\eta_0}{2}\right)\text{ and }\chi_\Phi(\mu_1)<\al<P_\Phi(\mu_1).$$
Apply the argument for Case \ref{c1}. We can find
$$\nu\in\spp_e(\al,\al')\cap B\left(\mu_1,\frac{\eta_0}2\right)
\subset\spp_e(\al,\al')\cap B(\mu_0,\eta_0).$$

\item\label{c3}
%\textbf{\flushleft Case 3.} 
Suppose that  $\chi_\Phi(\mu_0)<\al=P_\Phi(\mu_0)$.

There is $\eta_1\in\left(0,\dfrac{\eta_0}2\right)$ such that
$$\chi_\Phi(\mu_1)<\al\text{ for }
\mu_1:=\left(1-\frac{\eta_1}{D^*}\right)\mu_0+\frac{\eta_1}{D^*}\mu_M.$$
%\mu_0+\frac{\eta_1}{D^*}(\mu_M-\mu_0).$$
As $P_\Phi(\mu_0)=\al<P_\Phi(\mu_M)$, we have
$$\mu_1\in B\left(\mu_0,\frac{\eta_0}{2}\right)
\text{ and }\chi_\Phi(\mu_1)<\al<P_\Phi(\mu_1).$$
Apply the argument for Case \ref{c1}.
We can find
$$\nu\in\spp_e(\al,\al')\cap B\left(\mu_1,\frac{\eta_0}2\right)
\subset\spp_e(\al,\al')\cap B(\mu_0,\eta_0).$$

\item\label{c4}
%\textbf{\flushleft Case 4.}
Suppose that  $\chi_\Phi(\mu_0)=\al=P_\Phi(\mu_0)$ and
$\al\ge\chi_\Phi(\mu_M)$.

%Note that $P(\mu^P_{\max})>\al$.
Let
$$\mu_1:=\left(1-\frac{\eta_0}{3D^*}\right)\mu_0+\frac{\eta_0}{3D^*}\mu_M.$$
%\mu_0+\frac{\delta_0}{2D^*}(\mu_P-\mu_0).$$
As $\al=P_\Phi(\mu_0)<P_\Phi(\mu_M)$, we have %Then 
$$\mu_1\in B\left(\mu_0,\frac{\eta_0}2\right)\text{ and }\chi_\Phi(\mu_1)\le\al<P_\Phi(\mu_1).$$
Apply the argument for Case \ref{c1} if $\chi_\Phi(\mu_1)<\al$, 
or the one for Case \ref{c2} if $\chi_\Phi(\mu_1)=\al$.
We can find
$$\nu\in\spp_e(\al,\al')\cap B\left(\mu_1,\frac{\eta_0}2\right)
\subset\spp_e(\al,\al')\cap B\left(\mu_0,\eta_0\right).$$

\item\label{c5}
%\textbf{\flushleft Case 5.}
Suppose that  $\chi_\Phi(\mu_0)=\al=P_\Phi(\mu_0)$ and $\al<\chi_\Phi(\mu_M)$.

%In this case we have $h(\mu_0)=0$. 
%Let $\mu'\in\cm(X,f)$ be as in \eqref{algreat} such that
%Note that 
%\begin{equation}\label{algreat}
%$\chi_\Phi(\mu')<\al$.
%\end{equation}
As $h(f)>0$, there is $\mu'\in\cm(X,f)$ such that $h_{\mu'}(f)>0$.
Then there is $\kappa\in(0,1)$ such that
$$\chi_\Phi(\mu_1)<\al\text{ for }
\mu_1:=(1-\kappa)\mu_M+\kappa\mu'.$$
We also have $h_{\mu_1}(f)>0$. 
As $\chi_\Phi(\mu_1)<\al<\chi_\Phi(\mu_M)$,
there is $\lambda\in(0,1)$ such that
$$\chi_\Phi(\mu_2)=\al\text{ for }\mu_2:=\lambda\mu_1+(1-\lambda)\mu_M.$$
Again, we have $h_{\mu_2}(f)>0$. It follows that
$$P_\Phi(\mu_2)=h_{\mu_2}(f)+\chi_\Phi(\mu_2)>\al.$$
Let
$$\mu_3:=\left(1-\frac{\eta_0}{3D^*}\right)\mu_0+
\frac{\eta_0}{3D^*}\mu_2.$$
We have
$$\mu_3\in B\left(\mu_0,\frac{\eta_0}{2}\right), %h(\mu_3)>0, 
\chi_\Phi(\mu_3)=\al
\text{ and }P_\Phi(\mu_3)>\al.$$
Apply the argument for Case \ref{c2}. We can find
$$\nu\in\spp_e(\al,\al')\cap B\left(\mu_3,\frac{\eta_0}2\right)
\subset\spp_e(\al,\al')\cap B(\mu_0,\eta_0).$$
\end{enumerate}

\end{proof}

\begin{proposition}\label{presgen}
Suppose that $(X,f)$ is entropy-generic,
the entropy map is upper semi-continuous 
and $h(f)>0$.
Then $(X,f,\Phi)$ is pressure-generic.
\end{proposition}

\begin{proof}
Suppose that we are given
$\al\in\bigl(P_{\inf}(f,\Phi),P(f,\Phi)\bigr)$.
%Then $\spp^\al(X,f,\Phi)$ 
%Note that for any system $(X,f)$,   $\spp^\al$ is nonempty as long as 
%$$\al\in(P_{\inf}(f,\Phi),P(f,\Phi)).$$ 

There are $\mu_1,\mu_2\in\cm(X,f)$
such that 
$$P_\Phi(\mu_1)<\al<P_\Phi(\mu_2).$$ 
Then there is $\kappa\in(0,1)$
such that
$$P_\Phi(\mu_3)=\al\text{ for }\mu_3:=\kappa\mu_1+(1-\kappa)\mu_2.$$
It follows that $\chi_\Phi(\mu_3)\le P_\Phi(\mu_3)=\al$ and
$\mu_3\in\spp^\al(X, f, \Phi)$. This shows that $\spp^\al(X, f, \Phi)$
is nonempty.

As the entropy map is upper semi-continuous, so is $P_\Phi$.
Continuity of $\chi_\Phi$ and upper semi-continuity of $P_\Phi$
guarantee that
$$\spp^\al(X, f, \Phi)=\{\mu\in\cm(X,f):\chi_\Phi(\mu)\le\al\}
\cap \{\mu\in\cm(X,f):\al\le P_\Phi(\mu)\}$$
is a closed set in $\cm(X,f)$. So $\spp^\al(X, f, \Phi)$
is a compact subspace.

%Assume that $\al<\al'<P(f,\Phi)$.
Let $\al'>\al$.
It also follows from upper semi-continuity of $P_\Phi$
that 
$$\spp_{\al'}:=\{\mu\in\cm(X,f): P_\Phi(\mu)<\al'\}$$
is an open set in $\cm(X,f)$. 
As $\cm_e(X,f)$ is a $G_\delta$ subset of $\cm(X,f)$,
we have that
$$\spp_e(\al,\al')=\cm_e(X,f)\cap\spp_{\al'}\cap\spp^\al(X, f, \Phi)$$
is a $G_\delta$ subset of $\spp^\al(X, f, \Phi)$.
Hence by Lemma \ref{presdense}, $\spp_e(\al,\al')$
is residual in  $\spp^\al(X, f, \Phi)$.
It follows that
$$\spp_e(X, f, \Phi, \al)=%(\cm_e(X,f)\cap\cm^\al)\cap(
\bigcap_{k=1}^\infty \spp_e(\al,\al+\frac1k)$$
is also residual in $\spp^\al(X, f, \Phi)$.
\end{proof}

Conclusion \eqref{thit3}
of Theorem \ref{thmain}
follows from Conclusion \eqref{thit2} and Proposition \ref{presgen} when
$h(f)>0$. When $h(f)=0$, the approximate product property implies that
$(X,f)$ is uniquely ergodic \cite{Sununierg} and the conclusion holds
trivially.

\subsection{%Consequences of Pressure-Genericity
More Consequences
}

In this subsection we discuss two consequences of
pressure-genericity and entropy-genericity.
As corollaries of Theorem \ref{thmain},
Conclusions  \eqref{itnb2} and  \eqref{itnb3} of Corollary \ref{cor_nbhd}
follow from Proposition
\ref{localintpr} and Proposition \ref{pinfentzero},
respectively.

%the conclusions of both Proposition
%\ref{localintpr} and Proposition \ref{pinfentzero}
%hold for asymptotically entropy expansive systems
%with the approximate product property.

%Proposition \ref{localintpr} is an analog of Proposition \ref{localintent}.
%Proposition \ref{}

%We also have an analog of Proposition \ref{localintent}. 
%For $U\subset\cm(X,f)$, denote
%$$\cP(X,f,U):=\{h_\nu(f): \nu\in U\cap\cm_e(X,f)\}.$$
%As a consequence of entropy-genericity, there are ergodic measures
%of intermediate entropies in every neighborhood of an invariant measure.
%as shown in the following corollary.

Analogous to Proposition \ref{localintent}, pressure-genericity implies
that there are ergodic measures of intermediate pressures in every neighborhood.

\begin{proposition}\label{localintpr}
%For $U\subset\cm(X,f)$, denote
%$$\cP(X,f, \Phi, U):=\{P_\Phi(\nu): \nu\in U\cap\cm_e(X,f)\}.$$
If $(X,f)$ is pressure-generic, then
for every $\mu\in\cm(X,f)$ and every neighborhood $U$ of $\mu$, we have
        $$\begin{cases}
        \cP(X,f,U)\supset\bigl[\chi_\Phi(\mu), P_\Phi(\mu)\bigr],
         & \text{if }P_\Phi(\mu)<P(f,\Phi);\\
        \cP(X,f,U)\supset\bigl[\chi_\Phi(\mu), P_\Phi(\mu)\bigr), & \text{if }P_\Phi(\mu)=P(f,\Phi).
        \end{cases}$$
\end{proposition}

%The detailed proof is omitted.

\begin{proof}%[Proof the Theorem \ref{interent}]
There is $\eta>0$ such that $B(\mu,2\eta)\subset U$.

Suppose that 
$$\chi_\Phi(\mu)\le\al<P_\Phi(\mu)\le P(\Phi,f).$$
%$$
%There is $\mu_1\in\cm(X,f)$ such that $P_\Phi(\mu_1)>\al$.
%We fix 
%$$\ep:=\frac12\bigl(P_\Phi(\mu)-\al\bigr).$$
%By Proposition \ref{chicts}, there is $\eta'\in(0,\eta)$ such that
%$$\bigl|P_\Phi(\mu')-P_\Phi(\mu)\bigr|<\ep\text{ for every }\mu'\in B(\mu,\eta').$$
%Then $B(\mu,\eta)\cap\cm^\al(X,f)$ is a nonempty open set in
%$\cm^\al(X,f)$. 
%By  Theorem \ref{alresidual}
Then $\spp_e(X,f,\Phi,\al)$ is residual in $\spp^\al(X,f,\Phi)$. 
There is an ergodic measure
$$\nu\in B(\mu,\eta)\cap\spp^\al(X,f,\Phi)\subset U$$
such that $P_\Phi(\nu)=\alpha$. 
%Hence we have $\al\in\cH(X,f,U)$.
This implies that 
\begin{equation}\label{eqlocalpres}
\cP(X,f,U)\supset\bigl[\chi_\Phi(\mu), P_\Phi(\mu)\bigr).
\end{equation}

Suppose that $P_\Phi(\mu)<P(f,\Phi)$.
%, then
By Proposition \ref{ppvp}, there is $\mu_0\in\cm(X,f)$ such that 
$P_\Phi(\mu_0)>P_\Phi(\mu)$. Denote
$$\mu':=\left(1-\frac{\eta}{D^*}\right)\mu+\frac{\eta}{D^*}\mu_0
\in
B(\mu,2\eta)\subset
U.$$
%\text{ and }h_{\tilde\mu}(f)>h_\mu(f).$$
We still have $P_\Phi(\mu')>P_\Phi(\mu)$.
%and $\mu'\in U$.
%There is an ergodic measure
%$\nu\in B(\tilde\mu,\delta)\cap\cm^{h_\mu(f)}$ such that $h_\nu(f)=h_\mu(f)$.
%Then $\nu\in B(\mu,2\delta)\subset U$.
It follows from \eqref{eqlocalpres} that
$$\cH(X,f,U)\supset[0, h_{\mu'}(f))\supset[0, h_\mu(f)].$$
%In particular, by Proposition \ref{uppersemicts}, $(X,f)$ has ergodic measures
%of the maximal entropy.
\end{proof}

If $(X,f)$ is entropy-generic, then
$$\cm(X,f,0):=\bigl\{\mu\in\cm(X,f): h_\mu(f)=0\bigr\}$$
is dense in $\cm(X,f)$.
In this case, the infimum of $P_\Phi$
%either can not be obtained or
can only be obtained at 
a measure of zero entropy. 

\begin{proposition}\label{pinfentzero}
Suppose that $\cm(X,f,0)$ is dense in $\cm(X,f)$.
%$(X,f)$ is entropy-generic. 
%$\cm_e(X,f,0)$ is dense in $\cm(X,f)$
%and the entropy map is upper semi-continuous.
Then we have
        \begin{equation*}%\label{pinfchi}
        P_{\inf}(f,\Phi)=\chi^\Phi_{\min}.
        \end{equation*}
         %Suppose that 
%         As a corollary of \eqref{pinfchi}
         Moreover, if there is $\mu_0\in\cm(X,f)$ such that
        $P_\Phi(\mu_0)=P_{\inf}(f,\Phi)$,
        then 
        $$h_{\mu_0}(f)=0\text{ and }\chi_\Phi(\mu_0)=\chi^\Phi_{\min}.$$
\end{proposition}

\begin{proof}
%In fact, we only need the condition that the set
%$$\cm(X,f,0):=\bigl\{\mu\in\cm(X,f): h_\mu(f)=0\bigr\}$$
%is dense in $\cm(X,f)$. This follows from entropy-genericity.

Suppose that $P_{\inf}(f,\Phi)>\chi^\Phi_{\min}$. 
Let $\mu_{\min}^{\Phi}$ be as in Corollary \ref{corchimin}.
Then we have
\begin{align*}
h_{\mu_{\min}^{\Phi}}(f)=P_\Phi\left(\mu_{\min}^{\Phi}\right)-\chi_\Phi\left(\mu_{\min}^{\Phi}\right)
%\\&
\ge P_{\inf}(f,\Phi)-\chi^\Phi_{\min}>0.
\end{align*}
%As $(X,f)$ is entropy-generic and $\chi_\Phi$ is continuous
By Proposition \ref{chicts}, there is
a neighborhood $U$ of  $\mu_{\min}^{\Phi}$ such that
$$\left|\chi_\Phi(\nu)-\chi^\Phi_{\min}\right|<\frac12h_{\mu_{\min}^{\Phi}}(f)
\text{ for every }\nu\in U.$$
As $\cm(X,f,0)$ is dense, there is $\nu_0\in U$ with $h_{\nu_0}(f)=0$.
%Then we are led to the contradiction
It follows that
$$P_\Phi(\nu_0)=\chi_\Phi(\nu_0)+h_{\nu_0}(f)<P_{\inf}(f,\Phi),$$
which contradicts with the definition of $P_{\inf}(f,\Phi)$.

Now suppose that there is $\mu_0\in\cm(X,f)$ such that
        $P_\Phi(\mu_0)=P_{\inf}(f,\Phi)$.
%Then we have
It follows that
\begin{equation*}%\label{eqeqs}
h_{\mu_0}(f)=P_\Phi(\mu_0)-\chi_\Phi(\mu_0)
%=P_{\inf}(f,\Phi)-\chi_\Phi(\mu_0)
\le P_{\inf}(f,\Phi)-\chi^\Phi_{\min}=0\le h_{\mu_0}(f),
\end{equation*}
in which we must have equalities. This implies that
$h_{\mu_0}(f)=0$ and $\chi_\Phi(\mu_0)=\chi^\Phi_{\min}$.
\end{proof}

%It seems possible that there may be a sequence $\{\mu_n\}_{n=1}^\infty$
%of invariant measures of zero entropy such that
%$$\mu_n\to\mu_{\min}^{\Phi},\;
%\lim_{k\to\infty}
%P_{\Phi}(\mu_k)=P_{\inf}(f,\Phi)=\chi^\Phi_{\min}\text{ but }h_{\mu_{\min}^{\Phi}}(f)>0.$$
%In this case 
%It is possible that t
%We remark that there may not be any $\mu\in\cm(X,f)$ with
%        $P_\Phi(\mu)=P_{\inf}(f,\Phi)$.
%See Example \ref{expinf}.

\section{%Applications and 
Examples}\label{secexamp}

\begin{example}\label{nomaxent}
In \cite{Pavlov}, Pavlov showed that there is a subshift that has the tempered specification
property with a specified gap function and
has exactly two ergodic measures whose supports are disjoint.
In \cite{KOR}, Kwietniak, Oprocha and Rams constructed a one-sided shift $(X_2,\sigma)$ that has tempered
specification property and multiple but finitely many ergodic measures of maximal
entropy.
Both examples have the approximate product property. 
We note that if $\mu_1,\mu_2$ are two distinct ergodic measures of maximal entropy
for $(X_2,\sigma)$, then there is 
%$\eta>0$ such that the 
a neighborhood of $\frac{\mu_1+\mu_2}{2}$ that
%$B\left(\frac{\mu_1+\mu_2}{2},\eta\right)$ 
contains no ergodic measures of maximal
entropy. 
This indicates that the conclusion of Proposition \ref{localintent} is optimal.
%This example also shows that the approximate product property does not imply
%the decomposition introduced by Climenhaga and Thompson in \cite{CT}, since the latter.
\end{example}

\begin{example}\label{noappnothm}
The following example highlights the difference between
entropy-denseness and entropy-genericity.
Let $(X_0,\sigma_0)$ be a subshift constructed in \cite{HK} which is strictly ergodic and of positive
topological entropy. 
Let $\sigma$ be the full shift on $\Sigma:=\{0,1\}^{\NN}$.
%Let $X:=X_0\times\Sigma$ and $f:=
Let $(X,f)$ be the direct product of $(X_0,\sigma_0)$ and $(\Sigma,\sigma)$.
%the full shift. 
Then every
invariant measure for $(X,f)$ is a direct product of the 
unique ergodic measure for $(X_0,\sigma_0)$ and an
invariant measure for the full shift.
The system $(X,f)$ is expansive and entropy-dense, while
$\cm(X,f)$ is a Poulsen simplex.
But $(X,f)$ does not have ergodic measures whose entropies are less than $h(X_0,\sigma_0)$.
\end{example}

\begin{example}
\label{expinf}
Let $(\Sigma,\sigma)$ be the full shift and
$X_0$ be the closed subset of $\Sigma$, as in Example \ref{noappnothm},
that supports a unique ergodic measure $\mu_0$ with $h_{\mu_0}(\sigma)>0$.
Let $\varphi:\Sigma\to\RR$ be a continuous function such that
$$\begin{cases}
\varphi(x)=0, &\text{ if }x\in X_0;\\
\varphi(x)>0, &\text{ if }x\notin X_0.
\end{cases}$$
Let $\Phi:=\left\{\sum_{k=0}^{n-1}\varphi\circ \sigma^{k}\right\}_{n=1}^\infty$.
Then for every $\mu\in\cm(\Sigma,\sigma)$, no matter if $\mu$ is supported on $X_0$
or not, we always have $P_\Phi(\mu)>0$. However, Theorem \ref{thmain} applies
to $(X,\sigma)$ and there is a sequence $\{\nu_n\}_{n=1}^\infty$
of zero entropy that converges to $\mu_0$. It follows that
$$P_\Phi(\nu_n)=\chi_\Phi(\nu_n)\to\chi_\Phi(\mu_0)=0$$
and $P_{\inf}(\sigma,\Phi)=0$.
The infimum of $P_\Phi$ can not be assumed.
This example also shows that
%Theorem \ref{thmain} \eqref{thit3} and 
pressure-genericity may
fail for $\al=P_{\inf}(f,\Phi)$, as in this case
we have $\spp(\Sigma,\sigma,\Phi,0)=\emptyset$, which
is not generic in
$\spp^0(\Sigma,\sigma,\Phi)=\{\mu_0\}$.
\end{example}

%%%%%%%%%%%%%%
\iffalse
\begin{example}
\label{expinf1}
The following example show that Theorem \ref{thmain} \eqref{thit3}
and pressure-genericity may
fail for $\al=P_{\inf}(f,\Phi)$.
Let $(X,\sigma)$ be the full shift to which Theorem \ref{thmain} applies.
We can find two
disjoint closed subsets
$X_1$ and $X_2$ of $X$, which are both copies of the
set $X_0$ as in Example \ref{noappnothm},
such that $X_1$ supports a unique ergodic measure $\mu_1$,
$X_2$ supports a unique ergodic measure $\mu_2$ and
$h_{\mu_1}(\sigma)=h_{\mu_2}(\sigma)>0$.
Let $\varphi:X\to\RR$ be a continuous function such that
$$\begin{cases}
\varphi(x)=0, &\text{ if }x\in X_1\cup X_2;\\
\varphi(x)>0, &\text{ if }x\notin X_1\cup X_2.
\end{cases}$$
Let $\Phi:=\left\{\sum_{k=0}^{n-1}\varphi\circ f^{k}\right\}_{n=1}^\infty$.
Then we have $P_{\inf}(\sigma,\Phi)=0$ and
$\spp(X,f,\Phi,0)=\emptyset$
is not generic in
$$\spp^0(X,f,\Phi)=\bigl\{\lambda\mu_1+(1-\lambda)\mu_2:\lambda\in[0,1]\bigr\}.$$
%in which
\end{example}

\fi
%%%%%%%%%%%%%%%%%%%%%%

\begin{example}
\label{exnoneq}
The following is an example of a zero-entropy system that has
the approximate product property but is not minimal.
The idea % of the following example 
is due to Dominik Kwietniak.
See also \cite[Theorem 7.1]{Kwiet}.

Suppose that $X_1\subset\{0,1\}^{\NN}$ and 
%$\{c_n\}_{n=1}^\infty$ is a nondecreasing sequence of positive integers  
$L:\ZZ^+\to\ZZ^+$ be a tempered function
such that %$L(1)=1$
$$L(1)=1,\lim_{n\to\infty}{L(n)}=\infty$$
and for every $m\in\NN$, every $n\in\ZZ^+$, every $\{w_k\}_{k\in\NN}\in X_1$, we have
$$|\{m\le k<m+n:w_k=1\}|\le L(n).$$
For example, %if 
we may take
$$L(n)=\lfloor{1+\ln n}\rfloor\text{ for each }n.$$
%then $X_1$ is not a singleton.
The subshift on $X_1$ has the approximate product property. In fact, 
%if we apply 
under the metric
$$d(\{w_k\}_{k\in\NN},\{w'_k\}_{k\in\NN})=2^{-\min\{k:w_k\ne w'_k\}},$$
%then 
for every $\delta_1,\delta_2,\ep>0$,
there are $N\in\ZZ^+$ such that $2^{-N}<\ep$ and $M$ such that 
$$L({M+N})N<\delta_2M.$$ 
Then 
 every sequence $\sC$ in $X_1$ is $(n, \delta_1, \delta_2,
\{(k-1)M\}_{k=1}^\infty,
\ep)$-traced by the fixed point $\{0\}^{\NN}\in X_1$.
Moreover, we have
$$s(X_{1},n,\ep)\le Q\left(n+N,\frac{L({n+N})}{n+N}\right).$$
%This yields
Then by Lemma \ref{qndelta} one can show that $h(X_1,\sigma)=0$.
Moreover, $(X_1,\sigma)$ is a hereditary subshift. It is mixing
and non-invertible.
%The idea of %the following
%this example is due to Dominik Kwietniak.
%See also \cite[Theorem 7.1]{Kwiet}.
\end{example}

%\begin{example}
%\label{exnoneq}
%The following example comes from \cite[Theorem 7.l]{Kwiet}.
%Suppose that $X_1\subset\{0,1\}^{\ZZ}$ %and $\{c_n\}_{n=1}^\infty$ is a sequence of real numbers  
%such that
%$$\lim_{n\to\infty}c_n=0$$ and 
%for every $m\in\ZZ$, every $n\in\ZZ^+$, every $\{w_k\}_{k\in\ZZ}\in X_1$, we have
%$$|\{m\le k<m+2^n:w_k=1\}|\le n+1.$$
%For example, we can take
%$$c_n=\frac{1+\ln n}{n}.$$
%Then the subshift on $X_1$ has the approximate product property. In fact, 
%if we apply 
%under the metric
%$$d(\{w_k\}_{k\in\ZZ},\{w'_k\}_{k\in\ZZ})=2^{-\min\{|k|:w_k\ne w'_k\}},$$
%then 
%for every $\delta_1,\delta_2,\ep>0$,
%there are $N\in\ZZ^+$ such that $2^{-N}<\ep$ and $M$ such that 
%$$c_{M+N}(M+N)N<\delta_2M.$$ 
%Then every sequence $\sC$ in $X_1$ is $(n, \delta_1, \delta_2,
%\{0\}^{\ZZ^+},
%\ep)$-traced by the fixed point $\{0\}^{\ZZ}\in X_1$.
%Moreover, we have
%$$s_n(X_{1},n,\ep)\le Q(n+N,c_{n+N}).$$
%This yields that $h(X_1,\sigma)=0$.
%\end{example}

\begin{example}
\label{exnontrans}
Let $X_1\subset\{0,1\}^{\NN}$ and %$\{c_n\}_{n=1}^\infty$ 
$L:\ZZ^+\to\ZZ^+$
be as in Example \ref{exnoneq}.
Let $X_2\subset\{0,2\}^{\NN}$ such that for every $m\in\ZZ$, every $n\in\ZZ^+$, 
every $\{w_k\}_{k\in\ZZ}\in X_2$, we have
$$|\{m\le k<m+n:w_k=2\}|\le L(n).$$
Then the subshift on $X_1\cup X_2$  has the approximate product property as every sequence can be traced by $\{0\}^{\NN}$. The subshift is not topologically transitive and still has zero topological entropy. 
%It does not have the tempered gluing orbit property.
\end{example}

\begin{example}\label{nouniver}
Let $(X,f)$ be the direct product of the irrational rotation and the full shift. Then $(X,f)$ has the gluing
orbit property and is entropy expansive. Hence our results apply to it.  By \cite[Section 3.5]{QS}, this system $(X,f)$ is not universal.
This indicates that the intermediate entropy property is strictly weaker
than universality (which holds for systems
with the tempered specification property), 
hence should not be expected for broader classes of systems.
\end{example}

%%%%%%%%%%%%%
\iffalse
\begin{example}\label{tempernogo}
Let $(X,f)$ be the direct product of the irrational rotation and a quasihyperbolic
toral automorphism that is central skew, i.e. 
the associated matrix has non-diagonalizable
Jordan blocks with
unitary eigenvalues.
By \cite{Lind} and \cite{Marc}, such an automorphism has the tempered specification
property but not the specification
property.
Then $(X,f)$ has the tempered gluing
orbit property but it does not have the tempered specification property or the gluing orbit property. Our results still apply to it.  
%By \cite[Section 3.5]{QS}, this system $(X,f)$ is not universal.
\end{example}
\fi
%%%%%%%%%%%%%%%%

\begin{example}\label{noninvt}
Note that a 
systems with the approximate product property is not necessarily invertible.
The following is a non-symbolic example from \cite{BTV}.
Let $X:=\mathbb{T}^2$ and
$$f(x,y):=(2x \mod 1, x+y \mod 1)\text{ for every }(x,y)\in\mathbb{T}^2.$$
It is shown in \cite{BTV} that $(X,f)$ has the gluing orbit property. 
As $(X,f)$ is $C^\infty$, by \cite[Theorem 2.2]{Buzzi},
it is also asymptotically entropy expansive. %as it is $C^\infty$.
Our results apply to this system.
\end{example}

\section*{Acknowledgments}
The author is supported by National Natural Science Foundation of China (No. 11571387)
and CUFE Young Elite Teacher Project (No. QYP1902).
The author 
%appreciates the valuable comments from 
would like to thank
Weisheng Wu, Daniel J. Thompson, Dominik Kwietniak, Gang Liao
and Jian Li
for fruitful discussions and helpful comments.

%+Bibliography

%-Bibliography

\end{document}